%% file: DK2-2020-final.tex
\documentclass[a4paper,11pt]{article}
\usepackage{amsmath}
\usepackage{amssymb}
\usepackage{theorem}
\usepackage{euscript}
\usepackage{pstricks}
\usepackage{authblk}
\usepackage{relsize}

\usepackage{hyperref}
\hypersetup{
	colorlinks,
	citecolor=green,
	filecolor=black,
	linkcolor=magenta,
	urlcolor=magenta
}
\hypersetup{linktocpage}

\usepackage{color}

\def\II{{\mathbb I}}

\def\ZZ{{\mathbb Z}}
\def\NN{{\mathbb N}}
\def\RR{{\mathbb R}}

\def\IId{{\mathbb I}^d}

\def\NNd{{\mathbb N}^d}
\def\RRd{{\mathbb R}^d}

\def\NNn{{\mathbb N}_{-1}}
\def\NNdn{{\mathbb N}^d_{-1}}

\def\Ss{{\mathcal S}}

\def\Uas{\mathring{U}^{\alpha,d}_\infty}

\def\Lad{H^{\alpha}_{\infty}(\II^d)}

\def\supp{\operatorname{supp}}

\newtheorem{theorem}{Theorem}[section]
\newtheorem{lemma}[theorem]{Lemma}

\numberwithin{equation}{section}

\theoremstyle{definition}
\newtheorem{definition}[theorem]{Definition}

\theoremstyle{remark}

\newcommand*{\bigs}[1]{\scalebox{1.02}{\ensuremath#1}}

\newcommand{\Chi}{\raise .3ex
	\hbox{\large $\chi$}}

\newcommand{\R}{\mathbb{R}}

\newcommand{\N}{\mathbb{N}}

\usepackage[section]{algorithm}
\usepackage{algorithmicx}
\usepackage{algpseudocode}
\algrenewcommand\algorithmicrequire{\makebox[46pt][l]{\textrm{required:}}}
\algrenewcommand\algorithmicensure{\makebox[46pt][l]{\textrm{output:}}}
\algrenewcommand\algorithmicfunction{\textrm{function}}
\algrenewcommand\algorithmicwhile{\textrm{while}}
\algrenewcommand\algorithmicdo{}
\algrenewcommand\algorithmicend{\textrm{end}}
\algrenewcommand\algorithmicforall{\textrm{for all}}
\algrenewcommand\algorithmicfor{\textrm{for}}
\algrenewcommand\algorithmicrepeat{\textrm{repeat}}
\algrenewcommand\algorithmicuntil{\textrm{until}}
\algrenewcommand\algorithmicif{\textrm{if}}
\algrenewcommand\algorithmicthen{\textrm{then}}
\algrenewcommand\algorithmicelse{\textrm{else}}

\newcommand{\bb}{{\boldsymbol{b}}}

\newcommand{\bee}{{\boldsymbol{e}}}

\newcommand{\bk}{{\boldsymbol{k}}}
\newcommand{\bh}{{\boldsymbol{h}}}

\newcommand{\bs}{{\boldsymbol{s}}}

\newcommand{\bx}{{\boldsymbol{x}}}

\newcommand{\by}{{\boldsymbol{y}}}

\newcommand{\bz}{{\boldsymbol{z}}}

\newcommand{\bW}{{\boldsymbol{W}}}
\newcommand{\bN}{{\boldsymbol{N}}}

\newcommand{\bell}{{\boldsymbol{\ell}}}

\newcommand{\balpha}{{\boldsymbol{\alpha}}}

\newcommand{\bone}{{\boldsymbol{1}}}

\usepackage{todonotes}

\newcommand{\be}{\begin{equation}}
\newcommand{\ee}{\end{equation}}
\newcommand{\beq}{\begin{eqnarray}}
\newcommand{\beqq}{\begin{eqnarray*}}
\newcommand{\eeq}{\end{eqnarray}}
\newcommand{\eeqq}{\end{eqnarray*}}

\title{Computation complexity of deep ReLU  neural networks in high-dimensional approximation}
\author[a]{Dinh D\~ung}
\affil[a]{Vietnam National University, Hanoi, Information Technology Institute
	\protect\\
	144 Xuan Thuy, Cau Giay, Hanoi, Vietnam
	\protect\\
	Email: dinhzung@gmail.com
}

\author[b]{Van Kien Nguyen
	\footnote{Corresponding author: University of Transport and Communications, No.3 Cau Giay Street, Lang Thuong Ward, Dong Da District,
		Hanoi, Vietnam
		\protect\\
		{\it Email address}: kiennv@utc.edu.vn}
}
\affil[b]{Faculty of Basic Sciences, University of Transport and Communications
	\protect\\	No.3 Cau Giay Street, Lang Thuong Ward, Dong Da District,
	Hanoi, Vietnam
	\protect\\
	Email: kiennv@utc.edu.vn
}

\author[c]{Mai Xuan Thao}
\affil[c]{Department of Natural Sciences, Hong Duc University
	\protect\\
	565 Quang Trung, Thanh Hoa, Vietnam
	\protect\\ 
	Email: maixuanthao@hdu.edu.vn
}
\date{\today}
\begin{document}
	
\maketitle

\begin{abstract}
The purpose of the present paper is to study the  computation complexity of deep ReLU neural networks to approximate functions in H\"older-Nikol'skii spaces of mixed smoothness $\Lad$ on the unit cube $\IId:=[0,1]^d$. In this context, for any function $f\in \Lad$, we explicitly construct  nonadaptive and adaptive deep ReLU neural networks having an output that  approximates $f$  with a prescribed accuracy $\varepsilon$, and prove dimension-dependent bounds for the computation complexity of this approximation, characterized by  the  size and the depth of  this deep ReLU neural network, explicitly in $d$ and $\varepsilon$.  Our results show the advantage of the adaptive method of approximation by deep ReLU neural networks over nonadaptive one.


\medskip
\noindent
  
\end{abstract}

\section{Introduction}	

 Neural networks have been studied and used for more than 70 years, dating back to the foundational work of Hebb \cite{Heb49B} and of Rosenblatt \cite{Ros58}. In recent years, deep neural networks have been successfully applied to a striking variety of Machine Learning problems, including computer vision \cite{KSH12}, natural language processing \cite{WSC.16}, speech recognition and image classification \cite{LBH15}.  There has been a number of interesting papers that address the role of depth and architecture of deep neural networks in approximating sets of functions  which have a very special  regularity properties such as analytic functions \cite{EWa18,Mha96}, differentiable functions \cite{PeVo18,Ya17a}, oscillatory functions \cite{GPEB19}, functions in isotropic Sobolev or Besov spaces \cite{AlNo20,DDF.19,GKNV19,GKP20,Ya17b}, functions with dominating mixed smoothness \cite{MoDu19,Suzu18}  or in approximating solutions to partial differential equations \cite{GPR.20,OPS19, ScZe19}, to mention just a few. The main advantage of deep neural networks in approximation functions is that they can output compositions of functions cheaply and consequently improve the convergence rate of approximation error, see \cite{DDF.19, EWa18, Ya17a}. We refer the reader to recent surveys \cite{GPEB19, Pet20} for concept and results in deep neural network approximation theory.

The purpose of the present paper is to study  the  computation complexity of deep ReLU neural networks for approximation of H\"older-Nikol'skii functions having  mixed smoothness on the unit cube $\IId:=[0,1]^d$. 
Let us introduce the space $\Lad$ of our interest. For univariate functions $f$ on $\II:=[0,1]$, the difference operator $\Delta_h$ is defined by 
\begin{equation*}
\Delta_h f(x) := \
f(x + h) -  f(x),
\end{equation*}
for all $x$ and $h \ge 0$ such that $x, x + h \in \II$. 
If $u$ is a subset of $\{1,\ldots, d\}$, for multivariate functions $f$ on $\IId$
the  mixed  difference operator $\Delta_{\bh,u}$ is defined by 
\begin{equation*}
\Delta_{\bh,u} := \
\prod_{i \in u} \Delta_{h_i}, \quad \Delta_{\bh,\varnothing} = {\rm Id},
\end{equation*}
for all $\bx=(x_1,\ldots,x_d)$ and $\bh=(h_1,\ldots,h_d)$ such that $\bx, \bx +\bh \in \IId$. Here the univariate operator
$\Delta_{h_i}$ is applied to the univariate function $f$ by considering $f$ as a 
function of  variable $x_i$ with the other variables held fixed. 
If $0 < \alpha \le 1$, 
we introduce the semi-norm 
$|f|_{H^{\alpha}_{\infty}(u)}$ for functions $f \in C(\IId)$ by
\begin{equation} \nonumber
|f|_{H^{\alpha}_{\infty}(u)}:= \
\sup_{\bh > 0} \ \prod_{i \in u} h_i^{-\alpha}\|\Delta_{\bh,u}(f)\|_{C(\IId(\bh,u))}
\end{equation}
(in particular, $|f|_{H^{\alpha}_{\infty}(\varnothing)}= \|f\|_{C(\IId)}$), where 
$\IId(\bh,u):= \{\bx \in \IId: \, x_i + h_i \in \II, \, i \in u\}$.
The  H\"older-Nikol'skii space 
$\Lad$ of mixed smoothness $\alpha$ then is defined as the set of  functions $f \in C(\IId)$ 
for which the  norm 
\begin{equation*} 
\|f\|_{\Lad}
:= \ 
\max_{u \subset \{1,\ldots,d\}} |f|_{H^{\alpha}_{\infty}(u)}
\end{equation*}
is finite. 
 From the definition we have that $\Lad \subset C(\IId)$. Denote by  $\mathring{C}(\IId)$ the set of all functions $f \in C(\IId)$ vanishing on the boundary $\partial  \IId$ of $\IId$, i.e., the set of all functions $f \in C(\IId)$ such that  $f(x) =0$ if  $x_j=0$ or $x_j=1$  for some index $j \in \{1,\ldots,d\}$.
Denote by  $\Uas$ the set of all functions $f$ in the intersection 
$\mathring{H}^{\alpha}_\infty(\IId):=\Lad \cap \mathring{C}(\IId)$  such that $\|f\|_{\Lad} \le 1$. 

 Let us mention that function spaces having mixed smoothness appear naturally in many models of real world problem in mathematical physics, finance and other fields. For instance, in a recent work on regularity properties of solutions of the electronic Schr\"odinger equation, Yserentant \cite{Yser10}   has shown that the eigenfunctions of the electronic  Schr\"odinger operator have a certain mixed smoothness. 
Triebel \cite[Chapter 6]{Tri15B} has indicated a relation between 
Faber bases and sampling recovery in the context of  spaces with mixed smoothness and 
solutions of Navier-Stokes equations. In particular, when initial data  belongs to spaces with mixed
smoothness,  Navier-Stokes equations admit a unique  solution.
In mathematical finance, many problems are expressed as the expectation of some payoff function depending on quantities, such as stock prices, which are solutions of stochastic equations governed by Brownian motions.  The payoff function normally has kinks and jumps and belongs to a very high dimensional space. To approximate the expected value one can apply preliminary integration method with respect to a single well chosen variable to obtain a function of $d-1$ variables which  belongs to appropriate mixed Sobolev spaces in which Quasi-Monte Carlo can be applied efficiently, see \cite{GKLS18} and references therein. For a survey on various aspects of  high-dimensional approximation of functions having a mixed smoothness we refer the reader to the  book  \cite{DTU18B}.

In approximation theory, modern problems driven by   a lot of applications in  Information Technology, Mathematical Finance, Chemistry, Quantum Mechanics, Meteorology, and, in particular, in Uncertainty Quantification and Deep Machine Learning are being formulated in very high dimensions. Many times, numerical methods for such problems may demand computational cost increasing exponentially in dimension when the accuracy increases and as a consequence the method becomes  intractable when the dimension of input data is large. This phenomenon is called ``curse of dimensionality". 
Hence, the problem of estimating dimension-dependent  error  in high-dimensional approximation problems  arises naturally. Hyperbolic crosses and sparse grids promise to rid the ``curse of dimensionality" in some  problems when  high-dimensional data belongs to certain classes of functions having mixed smoothness. Approximation methods and sampling algorithms for functions having mixed smoothness constructed on  hyperbolic crosses and  sparse grids   give a surprising effect since hyperbolic crosses and sparse grids have the number of elements much less than those of standard domains and grids but give the same approximation error. This essentially reduces the computational cost, and therefore makes the problem tractable.

It is also observed that using deep neural networks to approximate functions from spaces of mixed smoothness can lessen or break the curse of dimensionality, see \cite{DN20arXiv,Suzu18}. There is a closed relation between approximating by deep ReLU networks and B-spline interpolation and quasi-interpolation representation, in particular by piecewise linear functions \cite{DDF.19,MoDu19,ScZe19,Suzu18,Ya17a,Ya17b}. Despite its simplicity, in deep neural networks, activation function Rectified Linear Unit (ReLU) $\max(0,x)$ is the most commonly used for its efficiency.

In the recent paper \cite{DN20arXiv}, we have studied  the approximation by  deep  ReLU neural networks, of functions from the H\"older-Zygmund space of mixed smoothness defined on the unit cube $\IId$ when the dimension $d$ may be very large.	
The approximation error  is measured in the norm of the   isotropic Sobolev space.  For any function $f$ from this space, we explicitly constructed a deep ReLU neural network having an output that  approximates $f$  with a prescribed accuracy $\varepsilon$, and proved tight dimension-dependent  estimates of the computation complexity of this approximation, characterized as  the size and  depth of  this deep ReLU neural network, explicitly in $d$ and $\varepsilon$. 

As a continuation of this paper  the present paper investigates  
nonadaptive and adaptive high-dimensional approximation  by  deep ReLU neural networks for functions from the classes  $\Uas$.
 The approximation error  is measured in the norm of $L_\infty(\IId)$. In this context, we pay attention on  the computation complexity of the deep ReLU networks, characterized by  the  size and depth of  this deep ReLU neural network, explicitly in $d$ and tolerance $\varepsilon$. A key tool for explicit construction of approximation methods by deep ReLU networks for functions in $\Lad$  is truncations of tensorized Faber series.

The main contribution of the present paper is as follows. 

Based on the decomposition of continuous functions by tensorized Faber series, for any $f\in \Uas$ we explicitly construct a deep ReLU neural network $\Phi_\varepsilon(f)$ having  the output  that approximates $f$ in  the $L_\infty(\IId)$-norm with a prescribed accuracy $\varepsilon$ and having  computation complexity expressing the dimension-dependent  size  
	\begin{equation} \label{eq-weight-01}
	W(\Phi_\varepsilon(f))\leq C_\alpha d   \bigg( \frac{K_1^d}{(d-1)!}\bigg)^{\frac{1}{\alpha}+1} \varepsilon^{-\frac{1}{\alpha}} \log(2\varepsilon^{-1})^{(d-1)(\frac{1}{\alpha}+1)+1}, 
\end{equation}
where $K_1=B^{1/(\alpha+1)}4\alpha^{-1} $ with $B=(2^\alpha-1)^{-1}$. The idea in proving the above result is to use  truncation  of Faber series  $R_n(f)$ as a intermediate approximation. Precisely,  we first approximate function $f\in \Uas$ by $R_n(f)$ and then approximate $R_n(f)$ by a deep ReLU network. 

The advantage of this method is that the deep ReLU neural networks are easily constructed and they have the same architecture for all functions in $\Uas$ , i.e., it is nonadaptive. However,
since this method  uses $R_n(f)$ as an intermediate approximation, a disadvantage of  it is that the computation complexity of deep ReLU networks is  not better than that when approximating functions  $f \in \Uas$ by the linear method $R_n(f)$.

To overcome this disadvantage we  develop a technique  used in \cite{Ya17b} and \cite{DDF.19} for  the univariate case. By this, we first represent    the difference $f- {R}_n(f)$ in a special form and then approximate terms in this  representation by deep ReLU networks. We emphasize that extension of technique in \cite{Ya17b} and \cite{DDF.19} to multivariate case and mixed smoothness is non-trivial task since one needs to construct a set of finite cardinality to approximate functions in  $\Uas$. For any $f\in \Uas$ we explicitly construct a deep ReLU neural network $\Phi_\varepsilon(f)$ of adaptive architecture having  the output  that approximates $f$ in  the $L_\infty(\IId)$-norm with a prescribed accuracy $\varepsilon$ and having   a size  estimated by 
\begin{equation} \label{eq-weight-02}
	W(\Phi_\varepsilon(f))\leq C_{\alpha} d^2\bigg(\frac{K_2^d}{(d-1)!}\bigg)^{\frac{2}{\alpha}+2} \varepsilon^{-
		\frac{1}{\alpha}}    \big(\log(2\varepsilon^{-1})\log  \log (2\varepsilon^{-1})\big)^{(1+\frac{1}{\alpha})(d-1)},
\end{equation}
where $K_2=4(2^{\alpha+3}B)^{\frac{1}{2\alpha+2}} (\alpha^{-1}\log(2\alpha^{-1}) )^{1/2}$. Comparing \eqref{eq-weight-01} and \eqref{eq-weight-02} we find the later estimation improves $\log (2\varepsilon^{-1})$.  Notice that the terms in  right-hand side of both \eqref{eq-weight-01} and \eqref{eq-weight-02} which depend on  dimension $d$ only decay as fast as super exponential in $d$.

\noindent
The outline of this  paper is as follows. In Section \ref{sec-dnn} we introduce necessary definitions and elementary facts on deep ReLU neural networks. Section \ref{sec-fabersystem} is devoted to recall a decomposition of continuous functions on the unit cube $\IId$ by Faber system  and approximation of functions  $f\in \Uas$ by truncations of Faber series $R_n(f)$ as well as by sets of finite cardinality. In  Section \ref{sec-linear-method} we explicitly construct  nonadaptive deep ReLU neural networks that approximate functions in $\Uas$ and prove  upper estimates for  size and the depth required. Section \ref{sec-adaptive-method} presents an improvement for approximation by adaptive deep ReLU neural networks of the results obtained in Section \ref{sec-linear-method}.  In Section \ref{sec-app} we give an application of our results in numerical approximation of solutions to elliptic partial differential equations. Conclusions are given in Section \ref{sec:conclusion}.

\noindent
{\bf Notation.} \ As usual, $\NN$ is the natural numbers, $\ZZ$ is the integers, $\RR$ is  the real numbers and $ \NN_0:= \{s \in \ZZ: s \ge 0 \}$; $\NNn= \NN_0\cup \{ -1\} $. 
The letter $d$ is  reserved for
the underlying dimension of $\RR^d$, $\NN^d$, etc. If $x\in \RR$, $\lfloor x\rfloor$ is defined to be the largest integer no larger than $x$. Vectorial quantities are denoted by boldface
letters and  $x_i$ denotes the $i$th coordinate 
of $\bx \in \RR^d$, i.e., $\bx := (x_1,\ldots, x_d)$. For $\bx\in \RR^d$, we denote 
	$|\bx|_p := \big(|x_1|^p+\ldots+|x_d|^p\big)^{1/p}$ if $0<p<\infty$, $|\bx|_\infty := \max_{j=1,\ldots,d}|x_j|$ and $\supp(\bx)=\{j: x_j\not=0\}$. For $\bx, \by \in \RR^d$,  $\bx \by $ denotes
the  Euclidean inner product of $\bx, \by$, and
$2^\bx := (2^{x_1},\ldots,2^{x_d})$. For $\bk, \bs \in \NNd_0$,  we denote $2^{-\bk}\bs := (2^{-k_1}s_1,\ldots,2^{-k_d}s_d)$.  We use the abbreviation: $L_\infty:= L_\infty(\IId)$ and $\|\cdot\|_\infty:= \|\cdot\|_{L_\infty}$.
Universal constants or constants depending on parameters $\alpha, d$ are denoted by $C$ or $C_{\alpha,d}$, respectively.  Values of constants $C$ and  $C_{\alpha,d}$ in general, are not specified except the case when they are precisely given, and may be different in various places.   $|A|$ denotes the cardinality of the finite set $|A|$.
\section{Deep ReLU neural networks}\label{sec-dnn}
In this section we introduce necessary definitions and elementary facts on deep ReLU neural networks.
There is a wide variety of neural network architectures and each of them is adapted to specific tasks. We only consider feed-forward  deep ReLU neural networks  for which  only connections between
neighboring layers are allowed. 

\begin{definition}\label{def:DNN}
	Let $d,L\in \NN$, $L\geq 2$, $N_0=d$,  and $N_1,\ldots,N_{L}\in \NN$. Let $\bW^\ell=(w^\ell_{i,j})$, $\ell=1,\ldots,L$, be $N_\ell\times N_{\ell-1}$ matrix, and $\bb^\ell =(b^\ell_j)\in \RR^{N_\ell}$.  
\begin{itemize}
	\item   A  neural network  $\Phi$ with input dimension $d$ and $L$ layers
		is  a sequence of matrix-vector tuples
		$$
		\Phi=\big((\bW^1,\bb^1),\ldots,(\bW^L,\bb^L) \big).
		$$
	We will use the following terminology.
	\begin{itemize}
		\item The number of layers $L(\Phi)=L$  is the depth of $\Phi$;
		\item $N_w(\Phi)=\max_{\ell=0,\ldots,L}\{ N_\ell\}$ is the width of $\Phi$; $\bN(\Phi)=(N_0,N_1,\ldots,N_L)$ the dimension of $\Phi$;
		\item The real numbers $w^\ell_{i,j}$ and $b^\ell_j$ are  edge and node weights of  $\Phi$, respectively;
		\item The number of nonzero weights  $w^\ell_{i,j}$ and $b^\ell_j$  is the  size of $\Phi$ and denoted by $W(\Phi)$;
		\item When $L(\Phi) \ge 3$, $\Phi$ is called a deep neural network, and otherwise, a shallow neural network.
	\end{itemize}
\item A neural network  architecture $\mathbb{A}$ with input dimension $d$ and $L$ layers is a neural network 
$$\mathbb{A}=\big((\bW^1,\bb^1),\ldots,(\bW^L,\bb^L) \big),$$
where elements of $\bW^\ell$ and $\bb^\ell$, $\ell=1,\ldots,L$, are in $\{0,1\}$. 
\end{itemize}
\end{definition}

The above defined networks are sometimes called standard networks to distinguish with networks  allowing for connections of neurons in non-neighboring layers. A deep neural network can be visualized in a graph. 
The graph associated with a  deep neural network $\Phi$ defined in Definition \ref{def:DNN}  consists of  $L+1$ layers which are numbered from $0$ to $L$. The $\ell$th layer has $N_\ell$ nodes which are numbered from 1 to $N_\ell$. If $w^\ell_{i,j}\not =0$, then there is an edge connecting the node $j$ in the layer $\ell-1$ to the node $i$  in the layer $\ell$. In Figure \ref{fig:neuralnetwork} we illustrate a deep neural network with input dimension 3 and 5 layers.

\input{Neuralnetwork-DK2-2020}

\begin{definition} 
	Given $L\in \N
	$, $L\geq 2$, and a deep neural network  architecture  $\mathbb{A}= \big((\overline{\bW}^1,\overline{\bb}^1), \ldots,\allowbreak (\overline{\bW}^L,\overline{\bb}^L) \big)$. We say that a neural network  $\Phi=\big((\bW^1,\bb^1),\ldots,(\bW^L,\bb^L) \big)$  has architecture $\mathbb{A}$ if 
	\begin{itemize}
		\item  $\bN(\Phi)=\bN(\mathbb{A})$
		\item   $\overline{w}^\ell_{i,j}=0$ implies $w^\ell_{i,j}=0$, $\overline{b}^\ell_i=0$ implies $b^\ell_i=0$ for all $i=1,\ldots,N_\ell$, $j=1,\ldots,N_{\ell-1}$, and $\ell=1,\ldots, L$.  Here $\overline{w}^\ell_{i,j}$ are entries of $\overline{\bW}^\ell$ and $\overline{b}^\ell_i$ are elements of $\overline{\bb}^\ell$, $\ell=1,\ldots,L$.
	\end{itemize}

	For a given deep neural network $\Phi=\big((\bW^1,\bb^1),\ldots,(\bW^L,\bb^L) \big)$, there exists  
	a unique deep neural network  architecture $\mathbb{A}= \big((\overline{\bW}^1,\overline{\bb}^1),\ldots,(\overline{\bW}^L,\overline{\bb}^L) \big)$ such that
	\begin{itemize}
		\item  $\bN(\Phi)=\bN(\mathbb{A})$
		\item   $\overline{w}^\ell_{i,j}=0$ $\Longleftrightarrow$ $w^\ell_{i,j}=0$, $\overline{b}^\ell_i=0$ $\Longleftrightarrow$ $b^\ell_i=0$ for all $i=1,\ldots,N_\ell$, $j=1,\ldots,N_{\ell-1}$, and $\ell=1,\ldots, L$.  
	\end{itemize}
	We call this architecture $\mathbb{A}$ the minimal architecture of $\Phi$ (this definition is proper in the sense that any architecture of $\Phi$ is also an architecture of $\mathbb{A}$.)
\end{definition}
{ A deep neural network is   associated with an activation function which calculates output at each node. The choice of activation function depends on the problem under consideration.}  In this paper we focus our attention on ReLU activation function defined by 
$\sigma(t):= \max\{t,0\}, t\in \R$.  We will use the notation 
$\sigma(\bx):= (\sigma(x_1),\ldots, \sigma(x_d))$ for $\bx \in \RRd$.
\begin{definition}\label{def:ReLu-network}
	A deep ReLU neural network   with input dimension $d$ and $L$ layers is a neural network  
	$$\Phi=\big((\bW^1,\bb^1),\ldots,(\bW^L,\bb^L) \big)$$
	in which the following
	computation scheme is implemented
	\begin{align*}
		\bz^0&: = \bx \in \RR^d,
		\\
		\bz^\ell &: = \sigma(\bW^{\ell}\bz^{\ell-1}+\bb^\ell), \ \ \ell=1,\ldots,L-1,
		\\
		\bz^L&:= \bW^L\bz^{L-1} + \bb^L.
	\end{align*}
	We call $\bz^0$ the input  and with an 
	ambiguity denote
	$\Phi(\bx):= \bz^L$ the output of $\Phi$  and in some places we identify a deep ReLU neural network  with its output.
\end{definition}

Several deep ReLU neural networks can be combined  to form a larger  deep ReLU neural network whose output is a linear combination or composition of outputs of  sub-networks. In the following, we introduce parallelization, concatenation and special construction.
\begin{lemma}[Parallelization]\label{lem:parallel}
	Let $N\in \NN$, $\Omega\subset \RR^d$ be a bounded set, $\lambda_j\in \RR$, $j=1,\ldots,N$. Let $\Phi_j$, $j=1,\ldots,N$ be deep ReLU neural networks with input dimension $d$. Then  we can explicitly construct a deep ReLU neural network  denoted by $\Phi$ so that 
	$$
	\Phi(\bx)
	=
	\sum_{j=1}^N\lambda_j \Phi_j(\bx),\quad \bx\in \Omega,
	$$ 
	with  $L(\Phi)=\max_{j=1,\ldots,N}\{L(\Phi_j)\}$ and 
	$$
	 W(\Phi)=\sum_{j=1}^NW(\Phi_j)+\sum_{j: L(\Phi_j)<L(\Phi)}(L(\Phi)-L(\Phi_j)+2)\leq 3N \max_{j=1,\ldots,N}{W(\Phi_j)}.
	$$ The network $\Phi$ is called the Parallelization network of $\Phi_j$, $j=1,\ldots,N$.
\end{lemma}

A proof of Lemma \ref{lem:parallel} can be found in \cite{DN20arXiv}. The last estimate in Lemma \ref{lem:parallel} is due to $2\leq L \leq \max_{j=1,\ldots,N}{W(\Phi_j)}$.

Another way to construct a ReLU network whose output is a linear combination of outputs of other ReLU networks is to use special networks. A special deep ReLU neural network with  input dimension $d$ can be defined as follows.  In each hidden layer a special role is reserved for $d$ first (top) nodes and the last (bottom) node. Concatenation of top $d$ nodes and the bottom node in each layer to the corresponding nodes in the next layer form $d+1$ parallel channels. The nodes in these $d+1$ channel are free of activation. The top $d$ parallel channels are called the source channels and just carry $\bx=(x_1,\ldots,x_d)$ forward.  The bottom channel is called collation channel. The nodes in the bottom channel are used to collect intermediate outputs by addition. This channel never feeds forward into subsequent calculation, it only accepts previous calculations. It has been shown in \cite{DN20arXiv} that if  $\Phi$ is a special deep ReLU neural network with input dimension $d$ depth $L$ and $\bx \in \IId$,  then there is a deep ReLU neural network  $\Phi'$ such that  
	\begin{equation}\label{eq-special-l}
		L(\Phi') = L(\Phi) 
	\end{equation}
	and $\Phi'(\bx) = \Phi(\bx)$. In view of the proof of \cite[Lemma 4.2]{DN20arXiv} we find only node weights in the collation channel of $\Phi$ and $\Phi'$ are  different. Therefore we deduce 
	\begin{equation}\label{eq-special-w}
		W(\Phi')\leq W(\Phi)+L(\Phi) \leq 2W(\Phi).
	\end{equation}

\begin{lemma}[Special Construction]\label{lem:special}
	Let $N\in \NN$, $\Omega\subset \RR^d$ be a bounded set, $\lambda_j\in \RR$, $j=1,\ldots,N$. Let $\Phi_j$, $j=1,\ldots,N$ be deep ReLU neural networks with input dimension $d$. Then  we can explicitly construct a deep special ReLU neural network  denoted by $\Phi$ so that 
	$$
	\Phi(\bx)
	=
	\sum_{j=1}^N\lambda_j \Phi_j(\bx),\quad \bx\in \Omega,
	$$ 
	with  $L(\Phi)=\sum_{j=1}^NL(\Phi_j)$ and $W(\Phi)\leq \sum_{j=1}^NW(\Phi_j)+(d+1)L(\Phi) $. 
\end{lemma}
An illustration of a special network $\Phi$ whose output is a linear combination of network  $\Phi_j$, $j=1,\ldots,N$  is given in Figure \ref{fig:concatenation}.
 
\input{Concatenation-DK2-2020}

%
\begin{lemma}[Concatenation]\label{lem-concatenation} Let $\Phi_1$ and $\Phi_2$ be two ReLU neural networks   such that output layer of $\Phi_1$ has the same dimension as input layer of $\Phi_2$. Then, we can explicitly construct a ReLU neural network $\Phi$ such that $\Phi(\bx)=\Phi_2(\Phi_1(\bx))$ for $\bx\in \RR^d$. Moreover we have
	$L(\Phi) = L(\Phi_1)+L(\Phi_2)$ and $W(\Phi)\leq 2W(\Phi_1)+2W(\Phi_2)$. 
\end{lemma}

A proof of the above lemma can be found in \cite{PeVo18}. The network $\Phi$ in this lemma is called the concatenation network of $\Phi_1$ and $\Phi_2$.

\section{{Approximation by sets of finite cardinality}}\label{sec-fabersystem}

In this  section we recall a decomposition of continuous functions on the unit cube $\IId$ by Faber series,  interpolation approximation by truncated Faber series  and by set of finite cardinality.  They are  a key tool for explicit construction of approximation methods by deep ReLU networks for functions in H\"older-Nikol'skii spaces of mixed smoothness. 

Let  $\varphi(x)\ = \ (1 - |x-1|)_+$, $x \in \RR$, be  the hat function (the piece-wise linear B-spline with knots at $0,1,2$), where 
$x_+:= \max(x,0)$ for $x \in \RR$. 
For $k\in \NNn$ we define the functions $\varphi_{k,s}$ by
\begin{equation}\label{eq:faber1}
\varphi_{k,s}(x):=
\varphi(2^{k+1}x - 2s), \quad k \geq 0,  \ s \in Z(k):=\{0,1,\ldots, 2^{k} - 1\},
\end{equation}
and
\begin{equation}\label{eq:faber2}
\varphi_{-1,s}(x) := \varphi (x - s + 1),\ s\in Z(-1):=\{0,1\}.
\end{equation}
For a univariate function $f$ on $\II$, 
$k \in \NNn$, and $s\in Z(k)$ we define
\begin{equation*} 
\lambda_{k,s}(f) \ := 
- \frac {1}{2} \Delta_{2^{-k-1}}^2 f\big(2^{-k}s\big), \ k \ge 0, \quad 
\lambda_{-1,s}(f) \ := f(s),
\end{equation*}
where 
$$
\Delta_h^2 f(x) := \
f(x + 2h) -  2f(x+h)+f(x),
$$
for all $x$ and $h \ge 0$ such that $x, x + h \in \II$. 
If $m\in \NN_0$ we put
\begin{equation}\label{R_m-1}
{R}_m(f) := \sum_{k=0}^m q_k(f),\qquad
q_k(f):=\sum_{s\in Z(k)}\lambda_{k,s}(f)\varphi_{k,s}.
\end{equation}
For $k\in \NN_0$, we define the functions $\varphi^*_{k,s} \in \mathring{C}(\II)$ by
	\begin{equation}\label{phi^*_ks}
	\varphi^*_{k,s}(x):=
	\varphi(2^{k+1}x - s + 1), \quad s \in Z_*(k):=\{1,\ldots, 2^{k+1} - 1\},
	\end{equation}
and for $f \in \mathring{C}(\II)$ one can check  
\begin{equation*} \label{R_m=Q_m}
{R}_{m}(f) 
\ = \
\sum_{s \in Z_*(m)} f(2^{-m-1}s)\varphi^*_{m,s}\,.
\end{equation*}
Hence ${R}_{m}(f) \in \mathring{C}(\II)$ interpolates $f$ at the points 
$2^{-m-1}s$, $s \in Z_*(m)$, that is, 
\begin{equation*} \label{R_m=f}
{R}_{m}(f)(2^{-m-1}s) 
\ = \
f(2^{-m-1}s), \quad 
s \in Z_*(m). 
\end{equation*}

Put $Z(\bk):={\mathlarger{\mathlarger{\mathlarger{\mathlarger{\times}}}}}_{j=1}^d Z(k_j)$.
For $\bk \in \NNdn$, $\bs \in Z(\bk)$, we introduce  the tensorized Faber   basis by
\begin{equation} \label{hat-function}
\varphi_{\bk,\bs}(\bx)
\, := \, \varphi_{k_1,s_1}(x_1)\cdot\ldots\cdot \varphi_{k_d,s_d}(x_d) ,\quad \bx\in \IId.
\end{equation}
We also define the linear functionals $\lambda_{\bk,\bs}$ for multivariate function $f$ on $\IId$, $\bk\in \NNdn$, and $\bs\in Z(\bk)$ by 
\begin{equation*} 
\lambda_{\bk,\bs}(f) \ := 
\prod_{i=1}^d \lambda_{k_i,s_i}(f),
\end{equation*}
where the univariate functional $\lambda_{k_i,s_i}$ is applied to the univariate function $f$ by considering $f$ as a function of variable $x_i$ with the other variables held fixed. We have the following lemma.


\begin{lemma} \label{lemma[convergence(d)]}
	The  tensorized Faber system
$
	 \big\{\varphi_{\bk,\bs}: \ \bk \in \NNdn, \bs\in Z(\bk)\big\}
$ is a basis in $C(\IId)$. Moreover,  every function 
	$f \in C(\IId)$ can be represented by the Faber series 
	\begin{equation} \label{eq:FaberRepresentation}
	f
		\ = \
	\sum_{\bk \in \NNdn} q_{\bk}(f), \ \qquad q_{\bk}(f) := \sum_{\bs\in Z(\bk)} \lambda_{\bk,\bs}(f)\varphi_{\bk,\bs}
	\end{equation}
	converging in the norm of $C(\IId)$.
\end{lemma}
 When $d=1$, the system \eqref{eq:faber1}, \eqref{eq:faber2} and above result goes back to Faber \cite{Fab09}. The decomposition \eqref{eq:FaberRepresentation} when $d=2$ and an extension for function spaces with mixed smoothness  was obtained  in \cite[Theorem 3.10]{Tri10B}. A generalization for the case $d\ge 2$ and also to B-spline interpolation and quasi-interpolation representation was established by the first author of this paper in \cite{Dung11a,Dung16}.

When $f \in \Uas$, $\lambda_{\bk,\bs}(f)=0$ if $k_j=-1$ for some $j\in \{1,\ldots,d\}$, hence we can write 
\begin{equation*}
f \ = \ \sum_{\bk \in \NNd_0} q_\bk(f)   
\end{equation*}
with unconditional convergence in $C(\IId)$, see \cite[Theorem 3.13]{Tri10B}.
In this case it holds the following estimate
	\begin{equation}\label{eq:lambda-estimate}
	\begin{aligned}
|\lambda_{\bk,\bs}(f)| &=  2^{-d}\bigg|\prod_{i = 1}^d
 \Delta_{2^{-k_i-1}}^2 f\big(2^{-\bk}\bs\big) \bigg|
 \\
&= 2^{-d}\bigg|\prod_{i = 1}^d
		\Big[\Delta_{2^{-k_i-1}} f\big(2^{-\bk}\bs+2^{-k_i-1}\bee^i\big)-\Delta_{2^{-k_i-1}} f\big(2^{-\bk}\bs \big)\Big] \bigg|
		\leq 2^{ -\alpha d}2^{- \alpha|\bk|_1},
	\end{aligned}
\end{equation}
for $\bk\in \NNd_0, \ \bs\in Z(\bk)$. Here
$\{\bee^i\}_{i=1,\ldots,d}$ is the standard basis of $\RR^d$.

For $f \in \mathring{C}(\IId)$,  we define the operator ${R}_m$ by 
\begin{equation*} \label{R_m^d}
{R}_m(f) 
:= \ 
\sum_{|\bk|_1 \leq m} q_\bk(f)
\ = \
\sum_{|\bk|_1 \leq m} \ \sum_{s \in Z(\bk)} \lambda_{\bk,\bs}(f)\varphi_{\bk,\bs}.
\end{equation*}
The  truncated Faber series  ${R}_m(f) \in \mathring{C}(\IId)$ completely determined by values  of $f$ at the points 
$2^{-\bk-\bone}\bs$, for $(\bk,\bs) \in G^d(m)$,
where
\[
G^d(m):= \big\{(\bk,\bs):\, \,   |\bk|_1 \le m,\ \bs \in Z_*(\bk) \big\},
\]
 $Z_*(\bk):= \prod_{j = 1}^d Z_*(k_j)$ and
$\bone=(1,\ldots,1)\in \NN^d$.
Moreover, ${R}_m(f)$ interpolates $f$ at the points 
$2^{-\bk-\bone}\bs$, for $(\bk,\bs) \in G^d(m)$, i.e.,
\begin{equation*} \label{R^d_m=f}
{R}_m(f)(2^{-\bk-\bone}\bs) 
\ = \
f(2^{-\bk-\bone}\bs), \quad 
(\bk,\bs) \in G^d(m).
\end{equation*}
The following lemma gives a $d$-dependent estimate of the  approximation error by $R_m(f)$ of $f \in \Uas$, see \cite{DN21}.

\begin{lemma} \label{thm-DT20}
	Let  $d \ge 2$, $m\in \NN$, and $0 < \alpha \le 1$. Then we have
\begin{equation*}  
		\begin{aligned}
			\sup_{f \in \Uas} \|f - {R}_m(f)\|_{\infty}
			&\le  
			2^{-\alpha} B^{d}\, 2^{- \alpha m} \, \binom{m+d}{d-1},\qquad B=(2^\alpha-1)^{-1}.
		\end{aligned}
	\end{equation*}
\end{lemma}

We make use the abbreviations: 
$\bx_j := (x_1,\ldots,x_j) \in \RR^j$;  $\bar{\bx}_j := (x_{j+1},\ldots,x_d) \in \RR^{d-j}$ with the convention $\bx_0 := 0$ for 
$\bx \in \RRd$ and $j = 0,1,\ldots, d-1$.
When $j=1$ we denote $x_1$ instead of $\bx_1$. 

For  $f \in \mathring{U}^{\alpha,1}_{\infty}$ we explicitly construct the function $S_f \in \mathring{C}(\II)$ by
\begin{equation} \label{approx:S_f}
	S_f := \ \sum_{s \in Z_*(m)} 2^{-\alpha (m+1)}l_s(f) \varphi^*_{m,s},
\end{equation}
where we put $l_0(f)=0$  and assign the values $ S_f(2^{-m-1}s) = 2^{-\alpha (m+1)}l_s (f)$ from left to right closest to $f(2^{-m-1}s)$ for $s = 1,\ldots, 2^{m+1}-1$. If there are two possible choices for $l_s(f)$ we choose $l_s(f)$ 
that is closest to the already determined $l_{s-1}(f)$.
We define 
\begin{equation} \label{Ss^{alpha}(m)}
	\Ss^{\alpha}(m) := \big\{S_f: f\in \mathring{U}^{\alpha,1}_{\infty} \big\}.
\end{equation}

It has been proved that the set $\Ss^{\alpha}(m) $ is finite and it holds the estimate $|\Ss^{\alpha}(m) |\leq 3^{2^{m+1}}$, see \cite{DN21}. Moreover, by Lemma \ref{thm-DT20} and \cite[Lemma 2.3]{DN21} for $f\in \mathring{U}^{\alpha,1}_{\infty}$ and $m\in \NN_0$ we have
\begin{equation}\label{eq-f-sf}
	\|f-S_f\|_{\infty} \leq 	\|f-R_m(f)\|_{\infty}+ \|R_m(f)-S_f\|_{\infty}\leq  2^{-(m+1)\alpha-\frac{1}{2}} + \frac{2^{-(m+1)\alpha}}{2^\alpha-1}  \,.
\end{equation}
In case of high dimensions we have the following.

\begin{lemma}\label{lem:pattern}
	Let   $m>1$, $d \ge 2$ and $0<\alpha \leq 1$.  For $f \in \Uas$, let the function ${S}_m(f)$ be defined by
\begin{equation} \label{S_f}
{S}_m(f)(\bx):= \sum_{|\bar{\bk}_1|\leq m }  2^{-\alpha (|\bar{\bk}_1|_1 + d - 1)}
\ \sum_{\bar{\bs}_1\in Z(\bar{\bk}_1)} 
\varphi_{\bar{\bk}_1,\bar{\bs}_1}(\bar{\bx}_1) S_{K_{\bar{\bk}_1,\bar{\bs}_1}(f)}(x_1),
\end{equation}
 where $S_{K_{\bar{\bk}_1,\bar{\bs}_1}(f)}\in \Ss^\alpha(m - |\bar{\bk}_1|_1)$  is  as in \eqref{approx:S_f} for the function $K_{\bar{\bk}_1,\bar{\bs}_1}(f)$. Then  it holds the inequality 	
\begin{equation*} \label{f-S<}
\|f-{S}_m(f)\|_{\infty} \leq   B^{d}  2^{-\alpha m} \binom{m+d}{d-1}.
\end{equation*}
 Moreover, for the set
\begin{equation*} \label{Ss^{d,alpha}(m)}
\Ss^{\alpha,d}(m):=\big\{{S}_m(f): \ f\in \Uas \big\},
\end{equation*}
we have 
 $N_{d}(m):=|\Ss^{\alpha,d}(m)| \le 3^{2^{m+1}\binom{m+d-1}{d-1}}$.
\end{lemma}

For a proof of the above lemma we refer the reader to \cite{DN21}.


\section{Deep ReLU network approximation -  a nonapdaptive method}\label{sec-linear-method}
In this section, we explicitly construct a   nonapdaptive deep ReLU neural network having an output that approximates  every  function  $f\in \Uas$ in  the  $L_\infty(\IId)$-norm  with a prescribed accuracy $\varepsilon$ and prove  dimension-dependent error estimates of  its size and depth. Nonadaptivity means that its architecture is the same for all $f\in \Uas$.
Our technique is first to approximate $f$ by its truncation of Faber series $R_n(f)$ and then approximate $R_n(f)$ by a deep ReLU network. Since the case $d=1$  was already  considered  (see, e.g.,  \cite{AlNo20,DDF.19,GKNV19} ), we study the high dimension case when  $d\geq 2$.  Our main result in this section is read as follows.

\begin{theorem}\label{thm-deep-NN1}
Let $d\in \NN$, $d\geq 2$ and $\alpha \in (0,1]$. Then there is $\varepsilon_0 =\varepsilon_0(d,\alpha)\in (0,1]$ such that for every $\varepsilon\in (0,\varepsilon_0)$ we can explicitly construct a deep neural network architecture ${\mathbb A}_\varepsilon$ with the following property.
For every $f \in \Uas$,  we can explicitly construct a deep ReLU neural network  $\Phi_\varepsilon(f)$ having the architecture ${\mathbb A}_\varepsilon$ such that	
	
	$$\|f-\Phi_\varepsilon(f)\|_\infty \leq \varepsilon.$$
	Moreover, we have
\begin{equation}\label{eq-thm1-size}
	W({\mathbb A}_\varepsilon)\leq C_\alpha d   \bigg( \frac{K_1^d}{(d-1)!}\bigg)^{\frac{1}{\alpha}+1} \varepsilon^{-\frac{1}{\alpha}} \log(2\varepsilon^{-1})^{(d-1)(\frac{1}{\alpha}+1)+1} 
\end{equation}
	and
	$$
	L({\mathbb A}_\varepsilon) \leq C \log d \log(2\varepsilon^{-1}),
	$$
	where $K_1=B^{1/(\alpha+1)}4\alpha^{-1}$ with $B$ given in Lemma \ref{thm-DT20} and $C_\alpha$ depends only on $\alpha$.
\end{theorem}

To prepare for proving Theorem \ref{thm-deep-NN1} we recall results of approximating the product $\prod_{j=1}^{d}x_j$ and $\varphi_{\bk,\bs}$ by deep ReLU neural networks, see \cite{ScZe19} and \cite{DN20arXiv}.

\begin{lemma} \label{lem:multi}
	For  every $\delta \in (0,1)$, $d\in \N$, $d\geq 2$, we can explicitly construct  a deep  ReLU neural network  $\Phi_P$  so that 
	$$
	\sup_{ \bx \in [-1,1]^d} \Bigg|\prod_{j=1}^d x_j - \Phi_P(\bx) \Bigg| \leq \delta. 
	$$
	Furthermore, if $x_j=0$ for some $j\in \{1,\ldots,d\}$ then $\Phi_P(\bx)=0$ and there exists a constant $C>0$ independent of $\delta$ and $d$ such that
	$$
	W(\Phi_P) \leq C d\log (d\delta^{-1}) 
	\qquad \text{and}\qquad
	L(\Phi_P)  \leq C\log d\log(d\delta^{-1}) \,.
	$$
\end{lemma}
\begin{lemma} \label{lem:product}
	For  every dimension $d\geq 2$, $\delta \in(0,1)$ and for the $d$-variate hat functions $\varphi_{\bk,\bs}$, $\bk \in \NNd_0$, $\bs \in Z(\bk)$,   defined as in \eqref{hat-function}, we can explicitly construct a deep neural network  $\Phi_\delta(\varphi_{\bk,\bs})$  so that 
	\begin{equation*} \label{varphi-Phi}
		\|\varphi_{\bk,\bs} - \Phi_\delta(\varphi_{\bk,\bs})\|_{\infty}
		\ \le \ \delta
	\end{equation*}
	and 
	\begin{equation} \label{N_w}
		W(\Phi_\delta(\varphi_{\bk,\bs})) \leq C d\log(d\delta^{-1})\quad \text{and}\quad L(\Phi_\delta(\varphi_{\bk,\bs})) \leq C\log d\log(d\delta^{-1})\,.
	\end{equation}
	Moreover,  $\supp  \Phi_\delta(\varphi_{\bk,\bs}) \subset \supp \varphi_{\bk,\bs}$.
\end{lemma}

The above result allows us to construct a deep ReLU network $\Phi_\varepsilon\big({R}_n(f)\big)$ to approximate ${R}_n(f)$.
\begin{lemma}\label{lem-Phi-Rnf} 
	Let  $d\in \NN$, $d\geq 2$, $n\in \NN$, $\alpha\in (0,1]$ and $ \varepsilon\in (0,1)$. Then for every $f\in \Uas$ we can explicitly construct a deep ReLU network $\Phi_\varepsilon\big({R}_n(f)\big)$  of the same architecture $\mathbb{A}_\varepsilon$ so that

	\begin{equation}\label{eq:approximate-S}
		\big\| {R}_n(f) - \Phi_\varepsilon\big({R}_n(f)\big) \big\|_\infty   \leq  \varepsilon.
	\end{equation}
Moreover, we have
\begin{equation}\label{eq-W(A-epsilon)}
	W\big(\Phi_\varepsilon({R}_n(f))\big)\leq  C d 2^{n}    \binom{n+d-1}{d-1} \log(d B^{d} \varepsilon^{-1})
\end{equation}
and
\begin{equation}\label{eq-L(A-epsilon)}
 L\big(\Phi_\varepsilon({R}_n(f))\big) \leq C  \log d   \log(d B^{d} \varepsilon^{-1}).
\end{equation}
The estimates \eqref{eq-W(A-epsilon)} and \eqref{eq-L(A-epsilon)} also hold for $W(\mathbb{A}_\varepsilon)$ and $L(\mathbb{A}_\varepsilon)$ respectively.
\end{lemma}
\begin{proof} For every pair $\bk, \bs $ with $|\bk|_1 \leq n$ and $\bs\in Z(\bk)$, by applying Lemma  \ref{lem:product} with $\delta :=  B^{-d} \varepsilon$, we explicitly construct 
	a deep ReLU neural network $\Phi_\delta(\varphi_{\bk,\bs})$	so that $\supp  \Phi_\delta(\varphi_{\bk,\bs}) \subset \supp \varphi_{\bk,\bs}$,
	\begin{equation} \label{varphi_{k,s} - Phi_delta}
		\| \varphi_{\bk,\bs} - \Phi_\delta(\varphi_{\bk,\bs}) \|_{\infty} \leq  B^{-d} \varepsilon,
	\end{equation}
	and it holds the estimates \eqref{N_w} for $W(\Phi_\delta(\varphi_{\bk,\bs}))$ and $L(\Phi_\delta(\varphi_{\bk,\bs}))$.
	We approximate ${R}_n(f)$ by  the  output
	\begin{equation*}\label{eq:phi-S}
		\Phi_\varepsilon\big({R}_n(f)\big)(\bx)  = \sum_{|\bk|_1 \leq n} \sum_{\bs\in Z(\bk)} \lambda_{\bk,\bs}(f) \Phi_\delta(\varphi_{\bk,\bs}) (\bx)
	\end{equation*}
	of the network $\Phi_\varepsilon\big({R}_n(f)\big)$ which is a parallelization of the networks $\{\Phi_\delta(\varphi_{\bk,\bs})\}_{|\bk|_1 \leq n,\, \bs\in Z(\bk)} $. Notice that the interiors of $\supp \Phi_\delta( \varphi_{\bk,\bs})$ and $\supp \Phi_\delta( \varphi_{\bk,\bs'})$ have empty intersection if $\bs \not= \bs'$. Moreover, for every $\bx \in \IId$, there is an $\bs \in Z(\bk)$ such that $\bx \in \supp \varphi_{\bk,\bs}$, and
	hence,  by using \eqref{eq:lambda-estimate} and  \eqref{varphi_{k,s} - Phi_delta} we get the estimates
	\begin{equation*}
		\begin{aligned}
			|{R}_n(f)(\bx) -  \Phi_\varepsilon\big({R}_n(f)\big)(\bx)| 
			&
			\ = \
			\sum_{|\bk|_1 \leq n} \big|\lambda_{\bk,\bs}(f) \big(\varphi_{\bk,\bs} (\bx)  -  \Phi_\delta(\varphi_{\bk,\bs})(\bx) \big)\big| 
			\\
			&
			\leq 
			2^{-\alpha d}\sum_{|\bk|_1\leq n} 2^{-\alpha |\bk|_1 }  \varepsilon B^{-d} 
			\\
			&\leq  \varepsilon(1-2^{-\alpha})^d\sum_{j=0}^n 2^{-\alpha j} \binom{j+d-1}{d-1}.
		\end{aligned}
	\end{equation*}
From
\begin{equation}\label{eq-DT20}
	\sum_{j=0}^\infty\binom{j+m}{m}t^j \leq (1-t)^{-m-1},\qquad t\in(0,1),
\end{equation} 
see \cite[Lemma 2.2]{DuTh20}, we  obtain \eqref{eq:approximate-S}.
	
	By using Lemma \ref{lem:parallel} and the estimates \eqref{N_w}, the size   and the depth of $\Phi_\varepsilon\big({R}_n(f)\big)$ can be estimated as
	\begin{equation*}\label{eq:PhiS-estimate}
		\begin{aligned}
			W(\Phi_\varepsilon\big({R}_n(f)\big))&\leq  C |\{(\bk,\bs): |\bk|_1\leq n,\bs \in Z(\bk)\}| \max_{|\bk|_1\leq n,\bs \in Z(\bk)} W(\Phi_\delta(\varphi_{\bk,\bs}))
			\\
			&=C\sum_{\ell=0}^n 2^{\ell} \binom{\ell+d-1}{d-1}d  \log(d B^{d} \varepsilon^{-1})
			\\
			&
			\leq  C d 2^{n}    \binom{n+d-1}{d-1} \log(d B^{d} \varepsilon^{-1}),
		\end{aligned}
	\end{equation*}
	and
	$$
	L(\Phi_\varepsilon\big({R}_n(f)\big)) 
	\leq \max_{|\bk|_1\leq n, \bs\in Z(\bk)} L(\Phi_\delta(\varphi_{\bk,\bs})) 
	\leq C \log d   \log(d B^{d} \varepsilon^{-1}).
	$$

	The proof is completed by noticing that $\Phi_\varepsilon\big({R}_n(f)\big)$ has the architecture $\mathbb{A}_\varepsilon$ (independent of $f$) which is defined as the minimal architecture of the deep ReLU neural network $\Phi_\varepsilon$ obtained by  parallelization of the networks $\{\Phi_\delta(\varphi_{\bk,\bs})\}_{|\bk|_1 \leq n,\, \bs\in Z(\bk)} $ with the output
$$		\Phi_\varepsilon(\bx)  = \sum_{|\bk|_1 \leq n} \sum_{\bs\in Z(\bk)}  \Phi_\delta(\varphi_{\bk,\bs}) (\bx),\quad \bx\in \IId.$$
Hence, the estimates \eqref{eq-W(A-epsilon)} and \eqref{eq-L(A-epsilon)} also hold for $W(\mathbb{A}_\varepsilon)$ and $L(\mathbb{A}_\varepsilon)$ respectively.
	\hfill
\end{proof}

We are ready to prove Theorem \ref{thm-deep-NN1}.

\begin{proof}Denote $n_0$ the natural point from which the function $h(n)=2^{-\alpha}B^d2^{-\alpha n} \binom{n+d}{d-1}$ is decreasing and $h(n-1)\leq 2^{-\alpha n/2} $ for all $n> n_0$. We put $\varepsilon_0=\min\{h(n_0),h(d)\}$. For $\varepsilon\in (0,\varepsilon_0)$ we define $n>\max\{n_0,d\}$ by  
\begin{equation}\label{eq-varepsilon}
	2^{-\alpha}B^{d}2^{-\alpha n} \binom{n+d}{d-1}\leq \frac{\varepsilon}{2} < 2^{-\alpha}B^{d}2^{-\alpha (n-1)} \binom{n-1+d}{d-1}.
\end{equation}
	With $\varepsilon'=\varepsilon/2$ in Lemma \ref{lem-Phi-Rnf} and $\Phi_\varepsilon(f)=\Phi_{\varepsilon'}({R}_n(f))$  we have
	\begin{align*} 
		\|f-\Phi_\varepsilon(f)\|_\infty \leq \|f-R_n(f)\|_\infty + \|{R}_n -  \Phi_{\varepsilon'}({R}_n(f)) \|_\infty \leq 2^{-\alpha}B^d2^{-\alpha n} \binom{n+d}{d-1}+\frac{\varepsilon}{2}\leq \varepsilon.
	\end{align*}
We define $\mathbb{A}_\varepsilon$ as the minimal architecture of the deep ReLU neural network $\Phi_\varepsilon$ obtained by  parallelization of the networks $\{\Phi_\delta(\varphi_{\bk,\bs})\}_{|\bk|_1 \leq n,\, \bs\in Z(\bk)} $ with the output
$$		\Phi_\varepsilon(\bx)  = \sum_{|\bk|_1 \leq n} \sum_{\bs\in Z(\bk)}  \Phi_\delta(\varphi_{\bk,\bs}) (\bx),\quad \bx\in \IId.$$
Then $\Phi_\varepsilon(f)$ has the architecture for all $f\in \Uas$. From Lemma \ref{lem-Phi-Rnf} we have
	\begin{equation*}\label{eq-W-phi-Rn}
W(\mathbb{A}_\varepsilon) \leq  C d 2^{n}  \log(2d\varepsilon^{-1}B^d)  \binom{n+d-1}{d-1}.
	\end{equation*}
	
	From the choice of $n$ we have
	$$
	2d\varepsilon^{-1}B^d \leq d2^\alpha 2^{\alpha n}\binom{n+d}{d-1}^{-1}\leq   2^{\alpha n} 2d\binom{n+d}{d-1}^{-1} \leq 2^{\alpha n}\,.
	$$
By this and \eqref{eq-varepsilon} we get
	\begin{align*}
		W(\mathbb{A}_\varepsilon)
		&
		\leq  C d \bigg(2\varepsilon^{-1}B^d\binom{n+d-1}{d-1}\bigg)^{1/\alpha} \alpha n  \binom{n+d-1}{d-1}
		\\
		&\leq  C d \big(2\varepsilon^{-1}B^d\big)^{1/\alpha}  n  \binom{n+d-1}{d-1}^{\frac{1}{\alpha}+1}
		\\
		&
		\leq C d \big(\varepsilon^{-1}B^d\big)^{1/\alpha}  n  \bigg( \frac{(2n)^{d-1}}{(d-1)!}\bigg)^{\frac{1}{\alpha}+1}.
	\end{align*}
Now $h(n-1)\leq 2^{-\alpha n/2} $  and \eqref{eq-varepsilon} lead to 
$ 
\frac{\varepsilon}{2}  \leq  2^{-\alpha n/2}$ which implies $n\leq \frac{2}{\alpha} \log(2\varepsilon^{-1})
$. Therefore we get
\begin{align*}
	W(\mathbb{A}_\varepsilon)
	& \leq Cd \big(\varepsilon^{-1}B^d\big)^{1/\alpha} \log(2\varepsilon^{-1}) \bigg( \frac{(4\alpha^{-1}\log(2\varepsilon^{-1}) )^{d-1}}{(d-1)!}\bigg)^{\frac{1}{\alpha}+1}
	\\
	& = Cd \big(B^d\big)^{1/\alpha} \bigg( \frac{(4\alpha^{-1} )^{d-1}}{(d-1)!}\bigg)^{\frac{1}{\alpha}+1} \varepsilon^{-\frac{1}{\alpha}} \log(2\varepsilon^{-1})^{(d-1)(\frac{1}{\alpha}+1)+1} 
\end{align*}
and \eqref{eq-thm1-size} follows. We also have
	$$
	L(\mathbb{A}_\varepsilon) 
	\leq C \log d \log\big(d2\varepsilon^{-1}B^d\big)\leq C \alpha n \log d  \leq C \log d \log(2\varepsilon^{-1}).
	$$
\hfill
\end{proof}

\section{Deep ReLU network approximation - an adaptive method}\label{sec-adaptive-method}
In this section,  we explicitly construct an adaptive method of approximation with accuracy $\varepsilon > 0$  by deep ReLU neural networks of functions $f \in \Uas$. This method  reduces the computation complexity expressing as the size and the depth of the 
	 approximating deep ReLU networks comparing with the computation complexity of the nonadaptive method given  in Theorem \ref{thm-deep-NN1}.  As mentioned the univariate case was already considered in \cite{DDF.19} ($0<\alpha<1)$) and \cite{Ya17b} ($\alpha=1$), we focus our attention on multivariate case when $d\geq 2$. The main result of this section is read as follows.

	\begin{theorem}\label{thm-deep-NN}
		Let $d\in \NN$,  $d\geq 2$, $\alpha \in (0,1]$. Then there is  $\varepsilon_0 =\varepsilon_0(d,\alpha)\in (0,1/2]$ such that for  every $\varepsilon\in (0,\varepsilon_0)$ and for  every $f\in \Uas$ we can explicitly construct an adaptive deep ReLU neural network $\Phi_\varepsilon(f)$ so that
		\begin{equation*} \label{|f(cdot)-}
		\|f-\Phi_\varepsilon(f)\|_{\infty} \leq \varepsilon.
		\end{equation*}
		Moreover, we have
		\begin{equation} \label{size}
		W(\Phi_\varepsilon(f))\leq C_{\alpha} d^2\bigg(\frac{K_2^d}{(d-1)!}\bigg)^{\frac{2}{\alpha}+2} \varepsilon^{-
			\frac{1}{\alpha}}    \big(\log(2\varepsilon^{-1})\log  \log (2\varepsilon^{-1})\big)^{(1+\frac{1}{\alpha})(d-1)}
		\end{equation}
		and
		\begin{equation} \label{depth}
		L(\Phi_\varepsilon(f)) \leq C'_\alpha \varepsilon^{-\frac{1}{d\alpha}} (\log (2\varepsilon^{-1}))^{\frac{d-1-\alpha}{d\alpha}}  (\log\log(2\varepsilon^{-1}))^{\frac{(\alpha+1)(d-1)}{d\alpha}} ,
		\end{equation}
		where 
		\begin{equation*} \label{C_{d,alpha}}
		\begin{aligned}
		K_2
		:=  4(2^{\alpha+3}B)^{\frac{1}{2\alpha+2}} (\alpha^{-1}\log(2\alpha^{-1}) )^{1/2}
		\end{aligned}
		\end{equation*}
		and positive constants $C_{\alpha}, C'_{\alpha}$ depend  on $\alpha$ only. 
	\end{theorem}
	
	Let us explain the idea of the proof. Let $f \in \Uas$ and $\varepsilon \in (0,\varepsilon_0)$ ($\varepsilon_0$ will be specified latter) be given.  Using the writing
		\begin{equation*} \label{f=}
f  \ = \ R_n(f) \ + \  (f - R_n(f)),
	\end{equation*}
we  explicitly construct deep ReLU neural networks to approximate with accuracy $\varepsilon/2$ the  terms $R_n(f)$  and  $f - R_n(f)$ and evaluate the dimension-dependent computation complexity separately, and then take their sum to get an approximation with accuracy $\varepsilon$ to $f$ and its dimension-dependent computation complexity.
For approximation of the first term ${R}_n(f)$, we take the deep ReLU neural network  $\Phi_{\varepsilon/2}({R}_n(f))$ which has been constructed in Lemma \ref{lem-Phi-Rnf}. 

Thus, our main task is to explicitly construct a desired deep ReLU neural network $\Phi_{\varepsilon/2}\big(f - {R}_n(f)\big)$ for approximation of the second term $f- {R}_n(f)$.
 Our strategy is  to represent    the difference $f- {R}_n(f)$ in a special form and then approximate terms in this  representation by deep ReLU networks. To this end, we need some auxiliary preparation.
 	
%
%

For univariate functions $f \in \mathring{C}(\II)$, let the operator ${T}_k$, $k \in \NN_0$, be defined by 
\[
{T}_k(f):= f - {R}_{k-1}(f)
\]
with the operator ${R}_k$ defined as in \eqref{R_m-1} and  the convention ${R}_{-1}:=0$. From this definition we have ${T}_0$ is the identity operator.  Notice that for $f \in \mathring{U}^{\alpha,1}_{\infty}$, it holds the inequality $\|{T}_k(f)\|_{H^{\alpha}_{\infty}(\II)} \le 2$.

For a multivariate function $f \in \mathring{C}(\IId)$, the tensor product operator  ${T}_\bk$, $\bk=(k_1,\ldots,k_d) \in \NN_0^d$, is defined by
\[
{T}_{\bk}(f) := \prod_{j=1}^d {T}_{k_j}(f),
\]
where the univariate operator ${T}_{k_j}$ is applied to the univariate function $f$ by considering $f$ as a function of variable $x_j$ with the other variables held fixed.

For $n\in \NN$, it has been shown in \cite{DN21} that  $f- {R}_n(f)$ can be represented in the following special form 
\begin{equation}\label{eq:representation}
	\begin{aligned}
		f-{R}_n(f)& 
		= \sum_{j = 0}^{d-1} \ \sum_{|\bk_j|_1 \leq n}F_{\bk_j},
	\end{aligned}
\end{equation}	
where $F_{\bk_0}
:= 
{T}_{(n+1)\bee^{1}}$ and
$$
F_{\bk_j}
:= 
{T}_{(n+1-|\bk_j|_1)\bee^{j+1}}\big(q_{\bk_j}(f)\big),\ \ j = 1,\ldots,d-1,
$$
 or equivalently,
\begin{equation}\label{F_k_j}
F_{\bk_j} 
=\prod_{i=1}^j \big({T}_{(k_i-1)\bee^j}-{T}_{k_i\bee^j}\big) {T}_{(n+1-|\bk_j|_1)\bee^{j+1}} (f), \ \ \ \ j = 1,\ldots,d-1.
\end{equation}	
We 
shall explicitly construct deep ReLU neural networks $\Phi_{{{\varepsilon}'}}(F_{\bk_j})$ to approximate each term $F_{\bk_j}$ in the sum in \eqref{eq:representation}.  Due to \eqref{F_k_j} this is reduced to construct deep ReLU networks that approximate $T_{\bk}(f)$, $\bk\in \NN_0^d$. Put 
\[
I_{\bk,\bs}
:= 
{\mathlarger{\mathlarger{\mathlarger{\mathlarger{\times}}}}}_{j=1}^d I_{k_j, s_j} 
= 
{\mathlarger{\mathlarger{\mathlarger{\mathlarger{\times}}}}}_{j=1}^d  [2^{-k_j}s_j,2^{-k_j}(s_j+1)], \ \bk \in \NN_0^d, \ \bs \in Z(\bk),
\]
and
$$
{T}_{\bk,\bs}(f)(\bx):=  2^{\alpha|\bk|_1 - d}\big({T}_{\bk}(f)\chi_{I_{\bk,\bs}}\big)\big( 2^{-\bk}(\bx+\bs)\big).
$$ 
Since $\supp \big({T}_{\bk}(f)\chi_{I_{\bk,\bs}}\big) \subset I_{\bk,\bs}$ and $\|{T}_{\bk}(f)\chi_{I_{\bk,\bs}}\|_{\Lad}\leq 2^d$, we have that
\begin{equation*} \label{T_k,s}
	\supp \big({T}_{\bk,\bs}(f) \big)\subset \IId,\qquad  {T}_{\bk,\bs}(f) \in \Uas.
\end{equation*}
Take the function ${S}_m({{T}_{\bk,\bs}(f)})$ defined as in \eqref{S_f} for ${T}_{\bk,\bs}(f) \in \Uas$.
By Lemma \ref{lem:pattern} it holds the estimate
\[
\big\|{T}_{\bk,\bs}(f)-{S}_m({{T}_{\bk,\bs}(f)})\big\|_{\infty} 
\leq 
B^{d}  2^{-\alpha m} \binom{m+d}{d-1}.
\] 
Define 
\begin{equation}\label{eq-S-Tkf}
	{S}_{\bk,m}(f)(\bx):= 2^{-\alpha |\bk|_1 + d}\sum_{\bs\in Z(\bk)} {S}_m\big({{T}_{\bk,\bs}(f)}\big)\big(2^\bk\bx-\bs\big).
\end{equation}
We then get
\begin{equation*}
	\begin{aligned}
		\big\| {T}_{\bk}(f) - {S}_{\bk,m}(f) \big\|_{\infty}
		& =\Bigg \| \sum_{\bs \in Z(\bk)} \Big[
		{T}_{\bk}(f)\chi_{I_{\bk,\bs}}(\cdot) -2^{-\alpha |\bk|_1 + d}  {S}_m\big({{T}_{\bk,\bs}(f)}\big)\big(2^\bk\cdot-\bs\big)\Big] \Bigg\|_{\infty}  \\
		&= 2^{-\alpha |\bk|_1 + d} \Bigg \| \sum_{\bs \in Z(\bk)} \Big[
		{T}_{\bk,\bs} (f) -  {S}_m\big({{T}_{\bk,\bs}(f)}\big)\Big]\big(2^\bk \cdot-\bs\big) \Bigg\|_{\infty} . 
	\end{aligned}
\end{equation*}
Since support of 
$ {T}_{\bk,\bs} (f) -  {S}_m\big({{T}_{\bk,\bs}(f)}\big)$ is contained in $ \IId$, we finally obtain
\begin{equation}\label{eq:Tkf}
	\begin{aligned}
		\big\| {T}_{\bk}(f) - {S}_{\bk,m}(f) \big\|_{\infty}
		&\leq   (2B)^{d} \big(2^{m} 2^{|\bk|_1}\big)^{-\alpha}\binom{m+d}{d-1}.
	\end{aligned}
\end{equation}

 Considering $ {S}_{\bk,m}(f)$ as an intermediate approximation of 
 $T_{\bk}(f)$ we shall construct deep ReLU networks approximating  $ {S}_{\bk,m}(f)$. Since ${S}_{\bk,m}(f)$ is a sum of  functions in $\Ss^{\alpha,d}(m)$, we shall construct a deep ReLU neural network $\Phi_\varepsilon(S)$ for  approximating  $S\in \Ss^{\alpha,d}(m)$ with accuracy $\varepsilon$ and estimate its size.

\begin{lemma}\label{lem-Phi-S}
Let  $d\in \NN$, $d\geq 2$, $m\in \NN$, $\alpha\in (0,1]$, and $ \varepsilon\in (0,1)$. Then for  every $S\in \Ss^{\alpha,d}(m)$, we can explicitly construct a deep ReLU neural network $\Phi_\varepsilon(S)$   so that $\supp  \Phi_\varepsilon(S)\subset \IId$ and 
\begin{align} \label{|S-Phi_S|}
\|S-\Phi_\varepsilon(S) \|_{\infty}  
& \leq    \varepsilon.
\end{align}
Moreover, there is a positive constant $C$ such that
\begin{align} \label{W(Phi_S)}
W(\Phi_\varepsilon(S))\leq Cd\log d 2^m \binom{m+d-1}{d-1} \log(d B^{d}\varepsilon^{-1}),
\quad
L(\Phi_\varepsilon(S)) 
\leq   C2^m\log d \log(d B^{d} \varepsilon^{-1}),
\end{align} 
 where $B$ is given in Lemma \ref{thm-DT20}.
\end{lemma}

\begin{proof}
By Lemma \ref{lem:pattern}, for   every function $S\in \Ss^{\alpha,d}(m)$, there is a function $f\in \Uas$ such that 
\begin{equation*} 		
S(\bx) =   
{S}_m(f)(\bx)= \sum_{|\bar{\bk}_1|_1\leq m }  2^{-\alpha (|\bar{\bk}_1|_1 + d - 1)}
\ \sum_{\bar{\bs}_1\in Z(\bar{\bk}_1)} 
\varphi_{\bar{\bk}_1,\bar{\bs}_1}(\bar{\bx}_1) S_{{\bar{\bk}_1,\bar{\bs}_1}}(x_1),
\end{equation*}
where 
$
S_{{\bar{\bk}_1,\bar{\bs}_1}}  := S_{K_{\bar{\bk}_1,\bar{\bs}_1}(f)} \in \Ss^{\alpha}(m - |\bar{\bk}_1|_1).
$
Since $S_{\bar{\bk}_1,\bar{\bs}_1}$ is a piecewise linear continuous function, see \eqref{Ss^{alpha}(m)} and \eqref{approx:S_f},
according to  \cite[Theorem 3.1]{DDF.19} we can explicitly construct a deep ReLU neural network $\Phi\big(S_{\bar{\bk}_1,\bar{\bs}_1}\big)$  with one-dimensional input so that 
$\Phi\big(S_{\bar{\bk}_1,\bar{\bs}_1}\big)(x_1)=S_{{\bar{\bk}_1,\bar{\bs}_1}}(x_1) $, $x_1 \in \II$,  and 
\begin{equation}\label{eq-W-Sks}
W\big(\Phi\big(S_{\bar{\bk}_1,\bar{\bs}_1}\big)\big)
\leq
C 2^{m-|\bar{\bk}_1|_1},\qquad L\big(\Phi\big(S_{\bar{\bk}_1,\bar{\bs}_1}\big)\big)
\leq 
C 2^{m-|\bar{\bk}_1|_1}.
\end{equation}
Each univariate function $\varphi_{k_j,s_j}$ in the tensor product 
$
 \varphi_{\bar{\bk}_1,\bar{\bs}_1} = \otimes_{j=2}^d\varphi_{k_j,s_j}
 $
  can be expressed as an output of a neural network $\Phi(\varphi_{k_j,s_j})$ with one-dimensional input, deep 3 and 8 weights.  Adding layers (with one node in each layer) putting forward $x_j$ to each network $\Phi(\varphi_{k_j,s_j})$ such that it has the length $L\big(\Phi\big(S_{\bar{\bk}_1,\bar{\bs}_1}\big)\big)$. We still denote these new networks by $\Phi(\varphi_{k_j,s_j})$. Then we obtain
$$
W(\Phi(\varphi_{k_j,s_j})) \leq C2^{m-|\bar{\bk}_1|_1}.
$$  
 We approximate the $d$-univariate function 
$ \varphi_{\bar{\bk}_1,\bar{\bs}_1}(\bar{\bx}_1)S_{{\bar{\bk}_1,\bar{\bs}_1}}(x_1)$ 
by the output of the network ${\Phi}_{\bar{\bk}_1,\bar{\bs}_1}$ with $d$-dimensional input which is explicitly constructed as a concatenation of the networks 
$\Phi\big(S_{\bar{\bk}_1,\bar{\bs}_1}\big)$, $\Phi(\varphi_{k_j,s_j})$, $j=2,\ldots,d$, with product network $\Phi_P$ in Lemma \ref{lem:multi}. With $\delta= \varepsilon  B^{1-d}$ in Lemma \ref{lem:multi} we have 
\begin{equation} \label{|varphi_}
 \big\|\varphi_{\bar{\bk}_1,\bar{\bs}_1}S_{{\bar{\bk}_1,\bar{\bs}_1}} - {\Phi}_{\bar{\bk}_1,\bar{\bs}_1}\big\|_{\infty}\leq  \varepsilon B^{1-d}.
\end{equation}
Since  $|\varphi_{\bar{\bk}_1,\bar{\bs}_1}(\bar{\bx}_1)|\leq 1 $ for $\bar{\bx}_1\in \II^{d-1}$ and $|S_{{\bar{\bk}_1,\bar{\bs}_1}}(x_1)|\leq 4$ for $x_1\in \II$ by \eqref{eq-f-sf}, from Lemmata \ref{lem-concatenation}, \ref{lem:multi}   and \eqref{eq-W-Sks}  we derive that
\begin{equation} \label{W(tilde{Phi}}
\begin{aligned}
W({\Phi}_{\bar{\bk}_1,\bar{\bs}_1})
&\leq C\Bigg(\sum_{j=2}^dW\big(\Phi(\varphi_{k_j,s_j})\big) + W\big(\Phi(S_{\bar{\bk}_1,\bar{\bs}_1})\big)+ W(\Phi_P)\Bigg)
\\
& \leq C d\big(2^{m-|\bar{\bk}_1|_1}  + \log(d B^{d}\varepsilon^{-1} )\big),
\end{aligned}
\end{equation}
and
\begin{equation}\label{L(tilde{Phi}}
	L({\Phi}_{\bar{\bk}_1,\bar{\bs}_1})\leq L\big(\Phi\big(S_{\bar{\bk}_1,\bar{\bs}_1}\big)\big) + L(\Phi_P)  \leq C\big( 2^{m-|\bar{\bk}_1|_1} + \log d \log(d B^{d}\varepsilon^{-1})\big).
\end{equation}
Moreover $\supp ({\Phi}_{\bar{\bk}_1,\bar{\bs}_1})\subset \supp (\varphi_{\bar{\bk}_1,\bar{\bs}_1}S_{{\bar{\bk}_1,\bar{\bs}_1}})$ by Lemma \ref{lem:multi}.

Let the network ${\Phi}_{\bar{\bk}_1}$  with output 
$$
{\Phi}_{\bar{\bk}_1}(\bx)= \sum_{\bar{\bs}_1\in Z(\bar{\bk}_1)}  {\Phi}_{\bar{\bk}_1,\bar{\bs}_1}(\bx)
$$ 
 be explicitly constructed as  a  combination of the networks $\big\{{\Phi}_{\bar{\bk}_1,\bar{\bs}_1} \big\}_{\bar{\bs}_1\in Z(\bar{\bk}_1)}$ by  the special construction. Then  by Lemma \ref{lem:special},  \eqref{W(tilde{Phi}} and \eqref{L(tilde{Phi}} we obtain that
\begin{equation} \label{L(tilde{Phi_k}}
	\begin{aligned}
L({\Phi}_{\bar{\bk}_1})
& \leq \sum_{\bar{\bs}_1\in Z(\bar{\bk}_1)}L({\Phi}_{\bar{\bk}_1,\bar{\bs}_1})
 \leq C 2^{|\bar{\bk}_1|_1} \big(2^{m-|\bar{\bk}_1|_1}  + \log d \log(d B^{d}\varepsilon^{-1} )\big)
 \\
 &
\leq 
C2^m\log d \log(d B^{d}\varepsilon^{-1}) 
\end{aligned}
\end{equation}
and
\begin{equation} \label{W(tilde{Phi_k}}
\begin{aligned}
	W({\Phi}_{\bar{\bk}_1})
	&\leq \sum_{\bar{\bs}_1\in Z(\bar{\bk}_1)}W({\Phi}_{\bar{\bk}_1,\bar{\bs}_1}) 
	+ 
(d+1)L({\Phi}_{\bar{\bk}_1})
	\\
&\leq C d 2^{|\bar{\bk}_1|_1} \big(2^{m-|\bar{\bk}_1|_1}  + \log(d B^{d}\varepsilon^{-1}) \big) + C(d\log d) 2^m \log(d B^{d}\varepsilon^{-1}) 
\\
&\leq C(d\log d) 2^m \log(d B^{d}\varepsilon^{-1}).
\end{aligned}
\end{equation}
Since $\bx \in \IId$, we can construct a standard network with the same output as ${\Phi}_{\bar{\bk}_1}$ and the estimates \eqref{L(tilde{Phi_k}} and \eqref{W(tilde{Phi_k}} hold, see \eqref{eq-special-l} and \eqref{eq-special-w}. We still denote this network by ${\Phi}_{\bar{\bk}_1}$. 
Now we define the network $\Phi_\varepsilon(S)$ as  a parallelization of the networks $({\Phi}_{\bar{\bk}_1})_{|\bar{\bk}_1|_1\leq m}$ with output
$$
\Phi_\varepsilon(S)(\bx)=\sum_{|\bar{\bk}_1|_1\leq m}2^{-\alpha (|\bar{\bk}_1|_1 + d - 1)}{\Phi}_{\bar{\bk}_1}(\bx).
$$
 Since $\supp \big({\Phi}_{\bar{\bk}_1,\bar{\bs}_1}\big)\subset \supp (\varphi_{\bar{\bk}_1,\bar{\bs}_1}S_{{\bar{\bk}_1,\bar{\bs}_1}})$ and for a given $\bar{\bk}_1$,
$\supp (\varphi_{\bar{\bk}_1,\bar{\bs}_1}S_{{\bar{\bk}_1,\bar{\bs}_1}})$ and $\supp (\varphi_{\bar{\bk}_1,\bar{\bs}'_1}S_{{\bar{\bk}_1,\bar{\bs}'_1}})$ are disjoint if $\bar{\bs}'_1 = \bar{\bs}_1$,
it holds by \eqref{|varphi_} and \eqref{eq-DT20}  that
\begin{align*}
\|S-\Phi_\varepsilon(S) \|_{\infty} 
& \leq  
\Bigg\|\sum_{|\bar{\bk}_1|_1 \leq m}\sum_{\bar{\bs}_1\in Z(\bar{\bk}_1)}2^{-\alpha (|\bar{\bk}_1|_1 + d - 1)}|\varphi_{\bar{\bk}_1,\bar{\bs}_1}S_{{\bar{\bk}_1,\bar{\bs}_1}} - {\Phi}_{\bar{\bk}_1,\bar{\bs}_1} |  \Bigg\|_{\infty}
\\
 &   =  
\sum_{|\bar{\bk}_1|_1 \leq m}2^{-\alpha (|\bar{\bk}_1|_1 + d - 1)}\max_{\bar{\bs}_1\in Z(\bar{\bk}_1)}\big\|\varphi_{\bar{\bk}_1,\bar{\bs}_1}S_{{\bar{\bk}_1,\bar{\bs}_1}} - {\Phi}_{\bar{\bk}_1,\bar{\bs}_1} \big\|_{\infty}
\\
& \leq \sum_{|\bar{\bk}_1|_1
	\leq m} 2^{-\alpha (|\bar{\bk}_1|_1 + d - 1)} \varepsilon B^{1-d} 
\\
&= \varepsilon  (1-2^{-\alpha})^{d-1} \sum_{\ell=0}^m 2^{-\ell \alpha} \binom{\ell+d-2}{d-2} \leq  \varepsilon.
\end{align*}
By Lemma \ref{lem:parallel} and \eqref{L(tilde{Phi_k}}, \eqref{W(tilde{Phi_k}} we obtain 
\begin{align*} 
W(\Phi_\varepsilon(S))\leq 3\big|\{\bar{\bk}_1:|\bar{\bk}_1|_1\leq m\}\big|\max_{|\bar{\bk}_1|_1\leq m}W({\Phi}_{\bar{\bk}_1})
&
 \leq C(d\log d)2^m \binom{m+d-1}{d-1}\log(d B^{d}\varepsilon^{-1}),
\end{align*} 
and
\begin{align*}
L(\Phi_\varepsilon(S))\leq \max_{|\bar{\bk}_1|_1\leq m}L({\Phi}_{\bar{\bk}_1})
&
\leq   C (\log d) 2^m\log(d B^{d}\varepsilon^{-1}).
\end{align*} 
Finally, the inclusion $\supp \Phi_\varepsilon(S) \subset \IId$ follows from Lemmata \ref{lem:multi} and \ref{lem:product}. 
\hfill
\end{proof}

The following result is a generalization of \cite[Lemma 5.1]{DDF.19} to $d$-dimensional case.

\begin{lemma}\label{lem:concatinate}
	Let $k\in \NN$, $\Lambda\subset Z(k)$ and $j\in \{1,\ldots,d\}$. Let $\Phi$ be a deep ReLU network with input dimension $d$ such that $\supp\Phi\subset  \IId$. Denote 
	$$
	f(\bx):=\sum_{s\in \Lambda} \Phi(x_1,\ldots,2^{k}x_j-s,\ldots,x_d),\quad \bx \in  \IId.
	$$  
	Then we can explicitly construct a deep ReLU network  $\Phi_\Lambda$ with output $f(\bx)$ and  
	\begin{equation}\label{eq-Phi-Lambda}
		W(\Phi_\Lambda)\leq C(d|\Lambda|+ W(\Phi)),\qquad L(\Phi_\Lambda)\leq 5+L(\Phi).
	\end{equation}
\end{lemma}
\begin{proof} Without loss of generality we assume that $j=1$. Set $H_{2^{k}-1}(t):= \sigma(t-2^{-k}s)/(1-2^{-k}s)$, $H_{2^k}(t) := 0$ and 
	$H_s:=\varphi_{k,s+1}^*$ for $s\in Z(k)\setminus\{ 2^k-1\}$, where  	$\varphi_{k,s+1}^*$ is defined as in \eqref{phi^*_ks}. Let 
	$$
	Z_i(k):=\{s\in Z(k): s=3r+i,\, r\in \N_0\},\ \ i=0,1,2.
	$$
	To make the proof simple, we divide it into  several cases of $\Lambda$ and $\Phi(\cdot)$.
 
	{\it Case 1.} The case $\Lambda\subset Z_i(k)$ for some $i\in \{0,1,2\}$ and $\Phi(\bx)\geq 0$ for all $\bx\in  \IId$. We will show that 
	\begin{equation}\label{eq-f(x)}
		f(\bx)=\sigma \Bigg(\Phi\bigg(\sum_{s\in \Lambda} H_s(x_1),\bar{\bx}_1\bigg) -\Phi\bigg(1-\sum_{s\in \Lambda} H_{s+1}(x_1),\bar{\bx}_1\bigg) \Bigg),
	\end{equation} 
	for all $\bx\in\IId$. Indeed, if $x_1\not \in \cup_{s\in \Lambda} [2^{-k}s,2^{-k}(s+3)]$ we have 
	$\sum_{s\in \Lambda} H_{s+1}(x_1)=  \sum_{s\in \Lambda} H_s(x_1)=0$. Since $\supp \Phi(\cdot) \subset \IId$  we get
	$$
	f(\bx)
	=0=
	\sigma\big(\Phi\big(0,\bar{\bx}_1\big) -\Phi\big(1,\bar{\bx}_1\big) \big).
	$$
	If $x_1\in [2^{-k}s_0,2^{-k}(s_0+1)]$ for some $ s_0\in \Lambda$ we have $\sum_{s\in \Lambda} H_{s+1}(x_1)= 0$ and $\sum_{s\in \Lambda} H_s(x_1)=2^{k}x_1-s_0$. Since $\Phi(\bx)\geq 0$ and $\supp \Phi(\cdot)\subset  \IId$ we obtain
	\begin{align*}
		f(\bx)
		=\Phi\big(2^{k}x_1-s_0,\bar{\bx}_1\big)=
		\sigma\big(\Phi\big(2^{k}x_1-s_0,\bar{\bx}_1\big) -\Phi\big(1,\bar{\bx}_1\big) \big) .
	\end{align*}  
	If $x_1\in [2^{-k}(s_0+2),2^{-k}(s_0+3)]$ for some $s_0\in \Lambda$ we have $\sum_{s\in \Lambda} H_s(x_1)=0$. Again from $\Phi(\bx)\geq 0$ and $\supp \Phi(\cdot)\subset						  \IId$ we get
	\begin{align*}
		f(\bx)=0=\sigma\bigg(\Phi\big(0,\bar{\bx}_1\big) -\Phi\bigg(1-\sum_{s\in \Lambda} H_{s+1}(x_1),\bar{\bx}_1\bigg) \bigg).
	\end{align*} 
	If $x_1\in [2^{-k}(s_0+1),2^{-k}(s_0+2)]$, $s_0\in \Lambda$, it is easy to see that $\sum_{s\in \Lambda} H_s(x_1)=1-\sum_{s\in \Lambda} H_{s+1}(x_1)$. Hence, the equality \eqref{eq-f(x)} holds. 	
	We have 
	$$
	H_s(x_1)=\sigma\big(1-\sigma\big(2^kx_1-s-1\big)-\sigma\big(s+1-2^{k}x_1\big)\big)
	$$
	for $s\in Z(k)\backslash\{2^k-1\}$ and $H_{2^{k}-1}(x_1)=\frac{1}{1-2^{-k}s}\sigma(x_1-2^{-k}s)$.

	Denote  the neural networks on the right side by $ \Phi(H_s)$. Then the functions $\sum_{s\in \Lambda} H_s(x_1)$ and $1-\sum_{s\in \Lambda} H_{s+1}(x_1)$ can be realized exactly by two networks $\Phi_1$ and $\Phi_2$ constructed by parallelization of $ \Phi_{H_s}$. By Lemma \ref{lem:parallel}, the length of $\Phi_1$ and $\Phi_2$ is 3 and their sizes are bounded $C|\Lambda|$. Since $\Phi_1(x_1)\geq 0$ and $\Phi_2(x_1)\geq 0$ when $x_1\in \II$, we can write
	$$f(\bx)=\sigma \big[\Phi\big(\sigma(\Phi_1(x_1)),\sigma(\sigma(\sigma(\bar{\bx}_1)))\big) -\Phi\big(\sigma(\Phi_2(x_1)),\sigma(\sigma(\sigma(\bar{\bx}_1)))\big) \big].$$
	Therefore, the network $\Phi_\Lambda$ is a concatenation of $\Phi_1$, $\Phi_2$, $ \sigma(\sigma(\sigma(\bar{\bx}_1)))$, and $\Phi$. 
	It is clear that we have the estimate
	$$
	W(\Phi_\Lambda)\leq C(d|\Lambda|+W(\Phi)) ,\qquad L(\Phi_\Lambda) \leq 4+L(\Phi). 
	$$
	
	{\it Case 2.} 
	The case $\Lambda\subset Z_i(\bk)$ for some $i\in \{0,1,2\}$ and $\Phi(\bx)$   changing sign when $\bx\in  \IId$. In this case, we write
	$\Phi(\bx)=\sigma\big(\Phi(\bx)\big) - \sigma\big(-\Phi(\bx)\big).$
	Hence
	$$
	f(\bx):=\sum_{s\in\Lambda  } \sigma \big(\Phi(2^{k}x_1-s,\bar{\bx}_1)\big)-\sum_{s\in\Lambda  }\sigma \big(-\Phi(2^{k}x_1-s,\bar{\bx}_1)\big),\quad \bx \in \IId.
	$$ 
	Applying the construction in Case 1 for each sum on the right side with $\Phi$ replaced by $ {\rm Id}(\sigma(\Phi(\cdot)))$ and ${\rm Id}(\sigma(-\Phi(\cdot)))$ respectively we obtain two neural networks $\Phi_\Lambda^+$ and $\Phi_\Lambda^-$. Here ${\rm Id}$ is the identity operator. Concatenating these two network by parallelization, see Lemma \ref{lem:parallel}, we obtain $\Phi_\Lambda$. Note that
	$$
	W\big({\rm Id}(\sigma(\Phi(\cdot)))\big)= W\big({\rm Id}(\sigma(-\Phi(\cdot)))\big) =W(\Phi)+1
	$$
	and
	$$
	L\big({\rm Id}(\sigma(\Phi(\cdot)))\big)= L\big({\rm Id}(\sigma(-\Phi(\cdot)))\big) =L(\Phi)+1.
	$$
	Therefore, the estimates \eqref{eq-Phi-Lambda} still hold true. 

	{\it Case 3.} General case.  We rewrite $f$ in the form:
	$$
	f(\bx)=\sum_{j=0,1,2}\sum_{s\in\Lambda\cap Z_j(k)} \Phi(2^{k}x_1-s, \bar{\bx}_1) .
	$$ 	
	To construct the network $\Phi_\Lambda$, we first  construct the network $\Phi_{\Lambda_j}$, $j=0,1,2$,  by using  the procedure in Case 2 to have that
	$$
	\Phi_{\Lambda_j}(\bx) = \sum_{s\in\Lambda\cap Z_j(k)} \Phi(2^{k}x_1-s, \bar{\bx}_1).
	$$
	Then by parallelizing  $(\Phi_{\Lambda_j})_{j=0,1,2}$ we obtain the network $\Phi_\Lambda$. From Lemma \ref{lem:parallel} we prove \eqref{eq-Phi-Lambda}.
	\hfill
\end{proof}

\begin{lemma} \label{prop-Phi-Skm}
Let $d,m\in \NN$, $d\geq 2$, $\bk\in \NN^d$, $\alpha\in (0,1]$ and $ \varepsilon\in (0,1)$. Assume that $\Phi_\varepsilon(S)$ is the neural network constructed in Lemma \ref{lem-Phi-S} to approximate $S\in \Ss^{\alpha,d}(m)$ with accuracy $\varepsilon$  and computation complexity as in \eqref{|S-Phi_S|} and \eqref{W(Phi_S)}. Then for  every $f\in \Uas$ we can explicitly construct a deep ReLU neural network $\Phi_\varepsilon\big({S}_{\bk,m}(f)\big)$  so that 
\begin{equation}\label{eq:approx-phi-skm}
	\big\| \Phi_\varepsilon\big({S}_{\bk,m}(f)\big) - {S}_{\bk,m}(f) \big\|_{\infty}
 \leq 2^{-\alpha |\bk|_1+d}  \varepsilon \,.
\end{equation}
Moreover,
\begin{equation}\label{eq:size-phi-skm}
\begin{aligned}
W( \Phi_\varepsilon\big({S}_{\bk,m}(f)\big))\leq Cd\bigg( 2^{|\bk|_1} +  \log d  2^{|\bk|_1-|\bk|_{\infty}}      N_d(m) 2^m \binom{m+d-1}{d-1} \log(d B^{d} \varepsilon^{-1})  \bigg), 
\end{aligned}
\end{equation}
and
\begin{equation}\label{eq:depth-phi-skm}
L\big( \Phi_\varepsilon\big({S}_{\bk,m}(f)\big)\big)\leq   C\log d \, N_d(m) 2^m\log(d B^{d} \varepsilon^{-1}),
\end{equation}
where $N_d(m)$ is given in Lemma \ref{lem:pattern} and  $B$ is given in Lemma \ref{thm-DT20}.
\end{lemma}

\begin{proof}
We can assume without loss of generality that $k_1=|\bk|_{\infty}$. By the definition \eqref{eq-S-Tkf}, for  $f \in \Uas$ we have that
$$  
{S}_{\bk,m}(f)(\bx) :=
2^{-\alpha |\bk|_1 +d }\sum_{\bar{\bs}_1\in Z(\bar{\bk}_1)} \sum_{s_1\in Z(k_1)}{S}_m\bigs({{T}_{\bk,\bs}(f)}\bigs)\big(2^\bk\bx-\bs\big).
$$
We number the elements of  the set $\Ss^{\alpha,d}(m)$  from $1$ to $N_{d}(m)$ as $S_1,\ldots,S_{N_{d}(m)}$. 
 For $\bar{\bs}_1\in Z(\bar{\bk}_1)$ and $\eta = 1,\ldots,N_{d}(m)$,  we define
$$
\Lambda_\eta(\bar{\bs}_1):= \big\{s_1 \in Z(k_1): {S}_m\big({{T}_{\bk,\bs}(f)}\big) = S_\eta \in \Ss^{\alpha,d}(m) \big\}.
$$
Hence, we can write
$$  {S}_{\bk,m}(f)(\bx)=
2^{-\alpha |\bk|_1 +d }\sum_{\bar{\bs}_1\in Z(\bar{\bk}_1)} \sum_{\eta=1}^{N_{d}(m)}\sum_{s_1\in \Lambda_\eta(\bar{\bs}_1)}  S_\eta\big(2^{\bk}\bx-\bs\big).
$$
To approximate 
$
{S}_{\bk,m}(f)
$
we use  the  output
\begin{equation}\label{eq-output-PhiSkm}
\Phi_\varepsilon\big({S}_{\bk,m}(f)\big)(\bx)
:=
2^{-\alpha |\bk|_1 +d }\sum_{\bar{\bs}_1\in Z(\bar{\bk}_1)} \sum_{\eta=1}^{N_{d}(m)}\sum_{s_1\in \Lambda_\eta(\bar{\bs}_1)}  \Phi_\varepsilon(S_\eta)(2^{\bk}\bx-\bs\big)
\end{equation} 
of a deep ReLU neural network $\Phi_\varepsilon\big({S}_{\bk,m}(f)\big)$. Let us first   show explicitly how to   construct such a network $\Phi_\varepsilon\big({S}_{\bk,m}(f)\big)$ and then estimate  its size and  depth.  Denote  by $\Phi_{S_\eta,\bar{\bs}_1}$ the network constructed by adding a layer of $d$ nodes before the input layer of $\Phi_\varepsilon(S_\eta)$. Computations at nodes in the first layer of $\Phi_{S_\eta,\bar{\bs}_1}$ are $\sigma(x_1)$ and $\sigma(2^{k_j}x_j-s_j)$, $j=2,\ldots,d$. Then by \eqref{W(Phi_S)} we have
$$ 
W(\Phi_{S_\eta,\bar{\bs}_1}) \leq 2(d-1)+1 + 
W(\Phi_\varepsilon(S_\eta)) \leq CW(\Phi_\varepsilon(S_\eta))
$$
and
$$
L(\Phi_{S_\eta,\bar{\bs}_1}) \leq 1+L(\Phi_\varepsilon(S_\eta))\,.
$$
Since $\supp\Phi_\varepsilon({S_\eta})\subset  \IId$, we have  $\Phi_{S_\eta,\bar{\bs}_1}(\bx)=\Phi_\varepsilon(S_\eta)\big(x_1,2^{\bar{\bk}_1}\bar{\bx}_1-\bar{\bs}_1\big)$. Hence we can write
$$  \Phi_\varepsilon\big({S}_{\bk,m}(f)\big)(\bx)=
2^{d-\alpha |\bk|_1  }\sum_{\bar{\bs}_1\in Z(\bar{\bk}_1)} \sum_{\eta=1}^{N_{d}(m)}\sum_{s_1\in \Lambda_\eta(\bar{\bs}_1)}  \Phi_{S_\eta,\bar{\bs}_1}\big(2^{k_1}x_1-s_1, \bar{\bx}_1\big).$$
Applying Lemma \ref{lem:concatinate}   to the function $\sum_{s_1\in \Lambda_\eta(\bar{\bs}_1)}  \Phi_{S_\eta,\bar{\bs}_1}\big(2^{k_1}x_1-s_1, \bar{\bx}_1\big)$, we can explicitly construct  a network $\Phi_{\Lambda_\eta(\bar{\bs}_1)}$ with the output
$$
\Phi_{\Lambda_\eta(\bar{\bs}_1)}(\bx)=\sum_{s_1\in \Lambda_\eta(\bar{\bs}_1)}  \Phi_{S_\eta,\bar{\bs}_1}\big(2^{k_1}x_1-s_1, \bar{\bx}_1\big)
$$
 so that  its size and depth satisfy
\begin{align*}
W\big(\Phi_{\Lambda_\eta(\bar{\bs}_1)}\big)
&
\leq  C \big(d|\Lambda_\eta(\bar{\bs}_1)|+ W(\Phi_{S_\eta,\bar{\bs}_1})\big)
\leq  C \big(d|\Lambda_\eta(\bar{\bs}_1)|+ W(\Phi_\varepsilon(S_\eta))\big)
\end{align*} 
 and
$$
L(\Phi_{\Lambda_\eta(\bar{\bs}_1)})\leq CL(\Phi_{S_\eta,\bar{\bs}_1})  \leq CL(\Phi_\varepsilon(S_\eta)).
$$
 
Let $\Phi_{\bar{\bs}_1}$ be the special network combining $(\Phi_{\Lambda_\eta(\bar{\bs}_1)})_{\eta =1,\ldots,N_{d}(m)}$ with output 
$$
\Phi_{\bar{\bs}_1}(\bx)=\sum_{\eta=1}^{N_{d}(m)}\Phi_{\Lambda_\eta(\bar{\bs}_1)}(\bx).
$$ 
By  Lemmata \ref{lem:special} and \ref{lem-Phi-S}  its length is bounded as
\begin{align*}
L(\Phi_{\bar{\bs}_1}) \leq \sum_{\eta=1}^{N_{d}(m)}L\big(\Phi_{\Lambda_\eta(\bar{\bs}_1)}\big) \leq C\sum_{\eta=1}^{N_{d}(m)} L(\Phi_\varepsilon(S_\eta)) \leq C \log d  N_{d}(m) 2^m\log(d B^{d} \varepsilon^{-1})
\end{align*}
and its size
\begin{align*}
W(\Phi_{\bar{\bs}_1})
&\leq 
\sum_{\eta=1}^{N_{d}(m)}W\big(\Phi_{\Lambda_\eta(\bar{\bs}_1)}\big) + (d+1)L(\Phi_{\bar{\bs}_1})
\\
&\leq 
\sum_{\eta=1}^{N_{d}(m)} C  \Big(d|\Lambda_\eta(\bar{\bs}_1)|+ W(\Phi_\varepsilon(S_\eta))\Big) + (d+1)L(\Phi_{\bar{\bs}_1})
\\
&\leq 
 C \Bigg(\sum_{\eta=1}^{N_{d}(m)}d|\Lambda_\eta(\bar{\bs}_1)|+ (d\log d)N_d(m)2^m  \log(d B^{d} \varepsilon^{-1})\binom{m+d-1}{d-1}\Bigg)
\\
& \leq Cd \bigg(
2^{k_1} +     (\log d)   N_d(m) 2^m \binom{m+d-1}{d-1} \log(d B^{d} \varepsilon^{-1}) \bigg)\,.
\end{align*}
Since $\bx \in \IId$, the network $\Phi_{\bar{\bs}_1}$ can be  
transformed to a standard ReLU neural network with the same output and estimation for depth and size (by adjusting the constants), see \eqref{eq-special-l} and \eqref{eq-special-w}. We still denote this new network by $\Phi_{\bar{\bs}_1}$.

The network $\Phi_\varepsilon\big({S}_{\bk,m}(f)\big)$ is a parallelization of $\big(\Phi_{\bar{\bs}_1}\big)_{\bar{\bs}_1\in \bar{\bk}_1}$ which has output \eqref{eq-output-PhiSkm}
and by Lemma \ref{lem:parallel}
\begin{align*}
W\big( \Phi_\varepsilon\big({S}_{\bk,m}(f)\big)\big)
&
\leq 2\cdot 2^{|\bar{\bk}_1|_1} \max_{\bar{\bs}_1\in \bar{\bk}_1}  W(\Phi_{\bar{\bs}_1})
\\
&
\leq 
C d 2^{|\bar{\bk}_1|_1}\bigg(
2^{k_1} +     \log d   N_d(m) 2^m \binom{m+d-1}{d-1} \log(d B^{d} \varepsilon^{-1}) \bigg)
\end{align*}
and
\begin{align*}
L\big( \Phi_\varepsilon\big({S}_{\bk,m}(f)\big)\big) \leq \max_{\bar{\bs}_1 \in Z(\bar{\bk}_1)}L(\Phi_{\bar{\bs}_1}) \leq  C\log d \, N_d(m) 2^m\log(d B^{d} \varepsilon^{-1}).
\end{align*} 
 Thus, \eqref{eq:size-phi-skm} and \eqref{eq:depth-phi-skm} have been proven.
Next, we prove the estimate of the approximation error  \eqref{eq:approx-phi-skm}.
 Notice that by the assumptions of the lemma and Lemma \ref{lem-Phi-S} $\supp S  \subset\IId$ and $\supp \Phi_\varepsilon(S)\subset \IId$ for all $S\in \Ss^{\alpha,d}(m)$, and it holds the estimate \eqref{|S-Phi_S|}. Moreover,  for different pairs $(\bs,\eta)$ and $(\bs',\eta')$, the supports of the functions 
$(\Phi_\varepsilon(S_\eta)-S_\eta)\big(2^\bk\cdot-\bs\big)$ and $(\Phi_\varepsilon(S_{\eta'})-S_{\eta'})\big(2^\bk\cdot-\bs'\big)$ are disjoint. Hence, by  \eqref{|S-Phi_S|} we obtain 
\begin{align*}
\big\| \Phi_\varepsilon\big({S}_{\bk,m}(f)\big) - {S}_{\bk,m}(f) \big\|_{\infty}
&= 2^{d-\alpha |\bk|_1   }\Bigg \| \sum_{\bar{\bs}_1\in Z(\bar{\bk}_1)} \sum_{\eta=1}^{N_{d}(m)}\sum_{s_1\in \Lambda_\eta(\bar{\bs}_1)}
(\Phi_\varepsilon(S_\eta)-S_\eta)\big(2^\bk\cdot-\bs\big)\Bigg\|_{\infty}
\\
&= 2^{d-\alpha |\bk|_1   }\max_{\bar{\bs}_1\in Z(\bar{\bk}_1)} \max_{1 \le \eta \le N_{d}(m)}\max_{s_1\in \Lambda_\eta(\bar{\bs}_1)}
\big \| (\Phi_\varepsilon(S_\eta)-S_\eta)\big(2^\bk\cdot-\bs\big)\big\|_{\infty}
\\
&
\leq 2^{d-\alpha |\bk|_1}  \varepsilon
\end{align*}
which proves \eqref{eq:approx-phi-skm}.
\hfill
\end{proof}

 We are now in position to prove Theorem \ref{thm-deep-NN}.

\begin{proof} For convenience, we divide the proof into several steps. 
	 
	{\it Step 1. [Construction a deep ReLU neural networks $\Phi_{\varepsilon}(f)$]}.
	Let us recall our plan of the proof. 
To approximate  $f \in \Uas$, we will construct a deep ReLU neural network with an output of the form
	\begin{equation} \label{Phi(f)=}
\Phi_{\varepsilon}(f)	 \ = \ \Phi_{\varepsilon/2}\big({R}_n(f)\big)  \ + \  \Phi_{\varepsilon/2}\big(f - {R}_n(f)\big),
	\end{equation}
	where $\Phi_{\varepsilon/2}\big({R}_n(f)\big)$ and   $\Phi_{\varepsilon/2}\big(f - {R}_n(f)\big)$ are deep ReLU neural networks approximating $R_n(f)$  and  $f - R_n(f)$ with accuracy $\varepsilon/2$, respectively.  
Then we have
	\begin{equation}\label{eq:f-r-app}
	\begin{aligned}
		\| f - \Phi_\varepsilon(f)\|_\infty 
		&\leq \big\| {R}_n(f) -  \Phi_{\varepsilon/2}\big({R}_n(f)\big) \big\|_{\infty} +  \big\|\big({f - R}_n(f)\big) - \Phi_{\varepsilon/2}\big({f - R}_n(f)\big)\big\|_\infty 
		\leq \varepsilon.
	\end{aligned}  
\end{equation}

For approximation of the first term ${R}_n(f)$, we take the deep ReLU neural network $\Phi_{\varepsilon/2}\big({R}_n(f)\big)$  which has been constructed in Lemma \ref{lem-Phi-Rnf}. Let us construct a deep ReLU neural network $\Phi_{\varepsilon/2}\big({f - R}_n(f)\big)$  for approximating $f - R_n(f)$ with accuracy $\varepsilon/2$.
	
	As noticed above, since  the difference  $f - R_n(f)$ is represented as in  \eqref{eq:representation}, we 
	shall explicitly construct deep ReLU neural networks $\Phi_{{{\varepsilon}'}}(F_{\bk_j})$ to approximate each term $F_{\bk_j}$ with  accuracy $\varepsilon'$ in the sum in \eqref{eq:representation}, where the value of $\varepsilon' $ will be chosen latter.
For ease of notation we consider the case $\supp(\bk_j) =j$  with $1\leq j\leq d-1$. The other cases are carried out similarly with a slight modification. From  \eqref{F_k_j}  we have 
\begin{align*}
F_{\bk_j} 
&
=\prod_{i=1}^j \big({T}_{(k_i-1)\bee^j}-{T}_{k_i\bee^j}\big) {T}_{(n+1-|\bk_j|_1)\bee^{j+1}} (f)
\\
&= \sum_{\bee \in \{0,1\}^j} (-1)^{|\bee|_1}  {T}_{\bk_j-\bee} {T}_{(n+1-|\bk_j|_1)\bee^{j+1}}(f)=\sum_{\bell\in \Lambda(\bk_j)} c_\bell T_\bell(f) \,,
\end{align*}
where
$$
\Lambda(\bk_j):=\Big\{\bell \in \NNd_0, \ \supp\bell\subset \{1,\ldots,j+1\}, \bell_j= \bk_j-\bee, \ell_{j+1}= n+1-|\bk_j|_1, \bee \in \{0,1\}^j\Big\}
$$
and $c_\bell$ is either $1$ or $-1$. It is easy to see that $|\Lambda(\bk_j)|\leq 2^j$ for all $\bk_j$ and if $\bell \in \Lambda(\bk_j)$ then $n+1-d\leq |\bell|_1\leq n+1$.

  We approximate $F_{\bk_j}$ by the output
\begin{equation*}
\Phi_{{{\varepsilon}'}}\big(F_{\bk_j}\big)(\bx):=  \sum_{\bell\in \Lambda(\bk_j)} c_\bell \Phi_{{{\varepsilon}'}} (S_{\bell,m}(f)),
\end{equation*} 
where  the networks $\Phi_{{{\varepsilon}'}}({S}_{\bell,m}(f))$  are constructed as in Lemma \ref{prop-Phi-Skm}. The network $\Phi_{{{\varepsilon}'}}\big(F_{\bk_j}\big)$ is a parallelization of  $\Phi_{{{\varepsilon}'}}({S}_{\bell,m}(f))$, $\bell \in \Lambda(\bk_j)$.

	We define  $\Phi_{\varepsilon/2}\big({f - R}_n(f)\big)$ as a deep ReLU neural network  with the output 
	\begin{equation}\label{eq-Phi-f}
	\Phi_{\varepsilon/2}\big({f - R}_n(f)\big):=  \sum_{j=0}^{d-1} \sum_{|\bk_j|_1\leq n}   \Phi_{{{\varepsilon}'}}\big(F_{\bk_j}\big)(\bx)
	\end{equation} 
	which is a 	parallelization of  $\Phi_{{ {\varepsilon}'}}\big(F_{\bk_j}\big)$, $|\bk_j|_1\leq n$, $j = 0, \ldots, d - 1$.
	It  approximates 	$f - R_n(f)$ with accuracy $\varepsilon/2$ by an appropriate choice of $\varepsilon'$.
	
	We put
	$$
  {{\varepsilon}'} =  B^{d}2^{-\alpha m}\binom{m+d}{d-1}
	$$
	with $m$ will be chosen later such that $\varepsilon' \in (0,1)$.
	We have from   \eqref{eq:Tkf} and Lemma \ref{prop-Phi-Skm}
\begin{equation*}\label{F_{k_j} - Phi}
	\begin{aligned}
		\big\| F_{\bk_j} - \Phi_{{{\varepsilon}'}}(F_{\bk_j}) \big\|_\infty 
		&\leq  \sum_{\bell \in \Lambda(\bk_j)} \Big( \big\| {T}_{\bell}(f)-{S}_{\bell,m}(f)\big\|_{\infty} + \big\| \Phi_{{ {\varepsilon}'}}\big(S_{\bell,m}(f)\big)-{S}_{\bell,m}(f)\big\|_{\infty}\Big)
		\\
	& \leq  \sum_{\bell \in \Lambda(\bk_j)}  
	\bigg(  (2B)^{d} \big(2^m 2^{|\bell|_1}\big)^{-\alpha}\binom{m+d}{d-1} + 2^{-\alpha |\bell|_1+d}    { {\varepsilon}'}\bigg)
	\\
	&
	\leq \sum_{\bell \in \Lambda(\bk_j)} 
	\bigg(  (2B)^{d}  \big(2^m 2^{n+1-d}\big)^{-\alpha}\binom{m+d}{d-1} + 2^{-\alpha (n+1-d)+d}    { {\varepsilon}'}\bigg)
	\\
	& \leq   2^{j+1-\alpha}(2^{\alpha+1} B)^{d}   \big(2^m 2^{n}\big)^{-\alpha}\binom{m+d}{d-1}. 
	\end{aligned}  
\end{equation*}
This leads to
\begin{equation*}
\begin{aligned}  
 \big\|\big({f - R}_n(f)\big) - \Phi_{\varepsilon/2}\big({f - R}_n(f)\big)\big\|_\infty 
&  \le \sum_{j=0}^{d-1} \sum_{|\bk_j|_1\leq n}  \big\|  \Phi_{{{\varepsilon}'}}\big(F_{\bk_j}\big)-F_{\bk_j}\big\|_{\infty}\\
 & \leq  \sum_{j=0}^{d-1} \sum_{|\bk_j|_1\leq n}  2^{j+1-\alpha}(2^{\alpha+1} B)^{d}   \big(2^m 2^{n}\big)^{-\alpha}\binom{m+d}{d-1}
 \\
  & \leq  \sum_{j=0}^{d-1} 
 2^j \binom{n+j}{j}  2^{1-\alpha}(2^{\alpha+1} B)^{d}   \big(2^m 2^{n}\big)^{-\alpha}\binom{m+d}{d-1}
 \\
 & \leq 2^{1-\alpha}(2^{\alpha+2} B)^{d}   \big(2^m 2^{n}\big)^{-\alpha}\binom{m+d}{d-1}\binom{n+d-1}{d-1}  =:A_{n,m}.
\end{aligned}  
\end{equation*}
In the next step, our task is to choose  $n,m$ (and therefore, $\varepsilon'$) depending on $\varepsilon$ such that $A_{n,m} \le \varepsilon/2$. Then we define the  deep ReLU neural network $\Phi_{\varepsilon}(f)$ as a parallelization of  the networks $ \Phi_{\varepsilon/2}\big({R}_n(f)\big) $ and $	\Phi_{\varepsilon/2}({f - R}_n(f))$ with the output \eqref{Phi(f)=}. From this \eqref{eq:f-r-app} follows. The size and depth of $\Phi_{\varepsilon}(f)$  are estimated explicitly in $d$ and $\varepsilon$ from the estimation of sizes and depths of $\Phi_{\varepsilon/2}\big({R}_n(f)\big)$ and $\Phi_{{{\varepsilon}'}}(F_{\bk_j})$ by the choice of $m,n$.

{\it Step 2. The choices of $\varepsilon_0$ and $n,\, m$.} Define $m_0\geq d$ as the smallest integer such that $   B^{d} 2^{-\alpha m_0}\binom{m_0+d}{d-1}<1$. Denote $n_0\in \NN$ from which the function 
	\begin{equation}\label{eq-h(n)}
	h(n):= K_{d,\alpha} 2^{-\alpha  n} n^{d-1-\alpha} (\log n)^{(\alpha+1)(d-1)}   , 
	\end{equation}
	where 
	\begin{equation}\label{eq-Kdalpha}
K_{d,\alpha}:=2 (2^{\alpha+2} B)^{d}  (4d \log 3)^\alpha \bigg(\frac{2^{d-1}}{(d-1)!}\bigg)^{\alpha+2}
	\end{equation}
	 is decreasing and $h(n-1)\leq 2^{-\alpha n/2}$ for all $n\geq n_0$. We put $n_1 = \lfloor (8d\log3)2^{m_0}\binom{m_0+d-1}{d-1} \rfloor +1$ and define $\varepsilon_0 =\min\{h(n_0),h(n_1),1/2\}$. For $\varepsilon\in (0,\varepsilon_0)$ we choose $n\in \NN$, $n\geq \max\{n_0,n_1\}$, such that $h(n)\leq \varepsilon/2<h(n-1)$ and then $m$ such that
	\begin{equation}\label{eq-Mdmm}
	( \log d)3^{2^{m+1}\binom{m+d-1}{d-1}} 2^m \binom{m+d-1}{d-1}m\leq  2^{\frac{n}{d}}<  	( \log d) 3^{2^{m+2}\binom{m+d}{d-1}} 2^{m+1} \binom{m+d}{d-1}(m+1).
	\end{equation}
These choices imply
	$$
	 3^{2^{m+1}\binom{m+d-1}{d-1}}  \leq  2^{\frac{n}{d}}<    3^{2^{m+3}\binom{m+d}{d-1}} 
	$$
and
	\begin{equation}\label{eq-m-n}
	2^{m+1}\binom{m+d-1}{d-1} \log 3 < \frac{n}{d} < (8 \log 3) 2^{m}\binom{m+d}{d-1}\quad \text{and}\quad m\leq \log n.
\end{equation}
Since  $n\geq (8d\log 3)2^{m_0}\binom{m_0+d}{d-1} $   we get $m\geq m_0\geq d$ and
\begin{align*}
A_{n,m}
& \leq   2^{1-\alpha}(2^{\alpha+2}B)^{d}    2^{-n\alpha}\bigg[dn^{-1}(8 \log 3) \binom{m+d}{d-1}\bigg]^{\alpha}\binom{m+d}{d-1} \binom{n+d-1}{d-1}
\\
& \leq 
2(2^{\alpha+2}B)^{d}       (4d \log 3)^\alpha   \bigg(\frac{2^{d-1}}{(d-1)!}\bigg)^{\alpha+2} 2^{-\alpha n } n^{d-1-\alpha} m^{(\alpha+1)(d-1)} \\
&   =h(n)
\leq \frac{\varepsilon}{2}.
\end{align*} 

{\it Step 3. Estimating the size and depth of $\Phi_{{{\varepsilon}'}} (S_{\bell,m}(f))$.} From  $n+1-d\leq |\bell|_1\leq n+1$ we have
$$
|\bell|_1-|\bell|_{\infty} \leq n+1-\frac{n+1-d}{d}\leq n-\frac{n}{d}+2
$$
which by 
\eqref{eq:size-phi-skm} leads to
\begin{equation*}
\begin{aligned}
W\big( \Phi_{{ {\varepsilon}'}}({S}_{\bell,m}(f))\big) 
& \leq  Cd\bigg( 2^{n} +  (\log d) 2^{n-\frac{n}{d}+m} \binom{m+d-1}{d-1} \log(d B^{d} {\varepsilon'}^{-1})  N_d(m) \bigg).
\end{aligned}
\end{equation*}
Note that by the choice of $\varepsilon'$  we get
$$
\log(d B^{d} {\varepsilon'}^{-1}) \leq \log\bigg(d 2^{\alpha m} \binom{m+d}{d-1}^{-1}\bigg) \leq \alpha m.
$$
It yields from \eqref{eq-Mdmm}
$$
(\log d)2^m \binom{m+d-1}{d-1}m N_d(m) \leq  3^{2^{m+1}\binom{m+d-1}{d-1}} 2^m \binom{m+d-1}{d-1}m\leq  2^{\frac{n}{d}}.
$$
Consequently
\begin{equation}\label{eq-phi-epsilon1}
	\begin{aligned}
		W\big( \Phi_{{{\varepsilon}'}}({S}_{\bell,m}(f))\big) 
		& \leq    C_\alpha d  2^{n}.
	\end{aligned}
\end{equation}
Similarly, we have
\begin{align*}
L\big( \Phi_{{{\varepsilon}'}}({S}_{\bell,m}(f))\big)
&\leq C \log d  \, N_d(m)2^m \log(d B^{d} {\varepsilon'}^{-1})
\leq C_\alpha (\log d)  3^{2^{m+1}\binom{m+d-1}{d-1}} 2^m m  \leq C_\alpha 2^{\frac{n}{d}} .
\end{align*}

{\it Step 4. Estimation of the size and depth of $\Phi_\varepsilon(f)$}.
 We recall that $\Phi_{\varepsilon/2}(f - R_n(f))$ is the network obtained by  parallelization of $\Phi_{{{\varepsilon}'}}({S}_{\bell,m}(f))$ with $\bell$ in the multi-set
$$
\Lambda=\big\{ \bell\in \Lambda(\bk_j),\ j=0,\ldots,d-1,\ |\bk_j|_1\leq n\big\} 
$$ and has the output equal to  the double sum on the right side of \eqref{eq-Phi-f}. We have
$$
|\Lambda| \leq  \sum_{j=0}^{d-1} \sum_{|\bk_j|_1\leq n} 2^j  =\sum_{j=0}^{d-1}   2^j\binom{n+j}{j} \leq 2^d\binom{n+d-1}{d-1}.
$$
The network $\Phi_{\varepsilon}(f)$ is a parallelization of $\Phi_{\varepsilon/2}({R}_n (f))$ and $\Phi_{\varepsilon/2}(f - R_n(f))$. Therefore, by Lemma \ref{lem:parallel} and the construction of  $\Phi_{\varepsilon/2}({R}_n (f))$ and  $\Phi_{\varepsilon/2}(f - R_n(f))$ we obtain
\begin{align*} 
W(\Phi_\varepsilon(f))
& \leq C\max\big\{W( \Phi_{\varepsilon/2}({R}_n (f))), W(\Phi_{\varepsilon/2}(f - R_n(f)) \big\}
\\
& 
\leq  C\max\bigg\{ W( \Phi_{\varepsilon/2}({R}_n (f))) \, , 2^d \binom{n+d-1}{d-1}\max_{\bell\in \Lambda}  W(  \Phi_{\varepsilon'}({S}_{\bell,m}(f)))\bigg\}.
\end{align*}

From Lemma \ref{lem-Phi-Rnf} and \eqref{eq-phi-epsilon1} we deduce that
\begin{align*} 
	W(\Phi_\varepsilon(f))
	&
	\leq  C_\alpha \max\bigg\{ d 2^{n}   \log(d B^{d}2 \varepsilon^{-1})  \binom{n+d-1}{d-1},  d 2^d 2^n\binom{n+d-1}{d-1}\bigg\}.
\end{align*}
Since $h(n)\leq \varepsilon/2$, from \eqref{eq-h(n)} by simple calculation we get
$$\log(dB^d\varepsilon^{-1}) \leq C_\alpha (d+ n) \leq C_\alpha dn $$
which implies
\begin{align*} 
	W(\Phi_\varepsilon(f))
\leq 
	C_\alpha d 2^d 2^{n}  n^d   \dfrac{2^{d-1}}{(d-1)!}\,.
\end{align*}
From $\varepsilon/2 \leq  h(n-1)\leq 2^{-\alpha n/2}$ we deduce  $   n \leq 2\alpha^{-1}\log( 2\varepsilon^{-1})$ and by $\eqref{eq-h(n)}$
\begin{equation}\label{eq-n}
2^{n}
\leq C_\alpha \Big(K_{d,\alpha} 2\varepsilon^{-1} n^{d-1-\alpha} (\log n)^{(\alpha+1)(d-1)} \Big)^{\frac{1}{\alpha}}.
\end{equation}
Consequently
\begin{align*} 
W(\Phi_\varepsilon(f))
& \leq  C_\alpha d2^d\frac{2^{d-1}}{(d-1)!} \Big(K_{d,\alpha} \varepsilon^{-1}  (n\log n)^{(\alpha+1)(d-1)} \Big)^{\frac{1}{\alpha}}
\\
&
\leq C_\alpha d 2^d  \frac{2^{d-1}}{(d-1)!}  {(K_{d,\alpha})}^{\frac{1}{\alpha}}  \varepsilon^{-\frac{1}{\alpha}} 
\Big(\log \big(2\alpha^{-1}\log( 2\varepsilon^{-1})\big) 2\alpha^{-1}\log( 2\varepsilon^{-1})\Big)^{(1+\frac{1}{\alpha})(d-1)}.
\end{align*}
We use the inequalities  with $p:=(1+\frac{1}{\alpha})(d-1)\geq 1$, $\varepsilon \in (0,1/2)$
\begin{equation}\label{eq-p}
\begin{aligned}
\big[\log \big(2\alpha^{-1}\log( 2\varepsilon^{-1})\big)\big]^{p}
& =\big[\log (2\alpha^{-1})+\log\log( 2\varepsilon^{-1})\big]^{p}
\\
&\leq \big[2\log \big(2\alpha^{-1})\big(\log\log( 2\varepsilon^{-1})\big)\big]^p
\end{aligned}
\end{equation}
to obtain
\begin{align*} 
	W(\Phi_\varepsilon(f))
	&
	\leq C_\alpha d 2^d  \frac{2^{d-1}}{(d-1)!}  {(K_{d,\alpha})}^{\frac{1}{\alpha}}  \varepsilon^{-\frac{1}{\alpha}} 
	\Big(4\alpha^{-1}\log (2\alpha^{-1})\log\log (2\varepsilon^{-1})\big)  \log( 2\varepsilon^{-1})\Big)^{(1+\frac{1}{\alpha})(d-1)}.
\end{align*}
Replacing $K_{d,\alpha}$ by the right-hand side of  \eqref{eq-Kdalpha}, we prove \eqref{size}.

Now we estimate the depth of $\Phi_\varepsilon(f)$. By applying Lemmata \ref{lem:special}, \ref{lem-Phi-Rnf} and by the construction of  $\Phi_{\varepsilon/2}(f - R_n(f))$  we have that
\begin{align*}
L(\Phi_\varepsilon(f)) 
&
=  C\max\big\{ L\big( \Phi_{\varepsilon/2}\big({R}_n(f)\big)\big),L\big( \Phi_{\varepsilon/2}(f - R_n(f)) \big) \big\}
\\
&
\leq C\max\big\{  \log d  \log(d B^{d} (\varepsilon/2)^{-1}),\max_{\bell\in \Lambda}L\big( \Phi_{\varepsilon'}({S}_{\bell,m}(f))\big) \big\}
\\
&\leq C_\alpha  \max\big\{\log d  \log(d B^{d} \varepsilon^{-1}) ,2^{\frac{n}{d}}\big\}
\\
&\leq C_\alpha  \max\big\{(d\log d)  n ,2^{\frac{n}{d}}\big\}.
\end{align*}
By \eqref{eq-Mdmm} and \eqref{eq-m-n} it is easily seen that $2^{\frac{n}{d}}$ dominates $(d\log d)  n$. 
From \eqref{eq-n}, the inequality $  n \leq 2\alpha^{-1}\log( 2\varepsilon^{-1})$, and \eqref{eq-p} we get
\begin{align*}
	2^{\frac{n}{d}}
	&
	\leq C_\alpha \Big(K_{d,\alpha} \varepsilon^{-1} n^{d-1-\alpha} (\log n)^{(\alpha+1)(d-1)} \Big)^{{1}/{(d\alpha)}}
	\\
& 	\leq C_\alpha \Big(K_{d,\alpha} \varepsilon^{-1}  (2\alpha^{-1}\log( 2\varepsilon^{-1}))^{d-1-\alpha} (\log (2\alpha^{-1}\log( 2\varepsilon^{-1})))^{(\alpha+1)(d-1)} \Big)^{{1}/{(d\alpha)}}
	\\
& 	\leq C_\alpha \Big(K_{d,\alpha} \varepsilon^{-1}  (2\alpha^{-1}\log( 2\varepsilon^{-1}))^{d-1-\alpha} (2\log (2\alpha^{-1})\log\log( 2\varepsilon^{-1})))^{(\alpha+1)(d-1)} \Big)^{{1}/{(d\alpha)}}\,.
\end{align*}
In view of \eqref{eq-Kdalpha}, we find that 
$$
\Big(K_{d,\alpha}   (2\alpha^{-1} )^{d-1-\alpha} (2\log (2\alpha^{-1}) )^{(\alpha+1)(d-1)} \Big)^{{1}/{(d\alpha)}}
$$
is bounded by a constant depending only on $\alpha$. Consequently
\begin{align*}
	L(\Phi_\varepsilon(f)) 
	&
 \leq C_\alpha   \varepsilon^{-\frac{1}{d\alpha}} (\log (2\varepsilon^{-1}))^{\frac{d-1-\alpha}{d\alpha}}  (\log\log(2\varepsilon^{-1}))^{\frac{(\alpha+1)(d-1)}{d\alpha}} 
\end{align*}
which proves \eqref{depth}.
\hfill
\end{proof}

\section{An application to numerical solving PDEs}\label{sec-app}

In this section, we apply the results on approximation by deep ReLU neural networks in Sections \ref{sec-linear-method} and \ref{sec-adaptive-method} for numerical approximation of the solution to elliptic PDEs. 	

Consider a modeled diffusion elliptic equation 
	\begin{equation*}
	- {\rm div} (a(\bx)\nabla u(\bx))
	\ = \
	f(\bx) \quad \text{in} \quad \IId,
	\quad u|_{\partial \IId} \ = \ 0, 
	\end{equation*}
	with a  function $f$ and a diffusion coefficient $a$ having sufficient regularity.
	Denote by $V:= H^1_0(\IId) = \mathring{W}^1_2(\IId)$ the energy space. If $a$ satisfies the ellipticity assumption
	\begin{equation} \nonumber
	0<a_{\min} \leq a(\bx) \leq a_{\max}<\infty,\ \forall \bx \in \IId,
	\end{equation}
	by the well-known Lax-Milgram lemma, there exists a unique solution $u \in V$ in weak form which satisfies the variational equation
	\begin{equation} \nonumber
	\int_{\IId} a(\bx)\nabla u(\bx) \cdot \nabla v(\bx) \, \text{d} \bx
	\ = \
	\int_{\IId}  f(\bx)  v(\bx) \, \text{d} \bx\,,  \quad \forall v \in V.
	\end{equation}
	
We want to approximate the solution $u$ by deep ReLU neural networks. The approximation error is measured in the norm of $L_\infty(\IId)$.
Assume for the modeled case that $a$ and $f$ have H\"older-Nikol'skii mixed smoothness $1$, i.e., $a,f \in H^1_\infty(\IId)$. Then, the solution $u$ has at least mixed derivatives $\partial^{\balpha}u$ with $\balpha\in \NN_0^d$, $\max_{j=1,\ldots,d}\alpha_j \leq 1$,  belonging to $L_2(\IId)$ \cite{GK09}, and therefore, by embedding for function spaces of mixed smoothness, see \cite[Theorem 2.4.1]{ST87B}, $u$ belongs to  $\mathring{H}^{1/2}_\infty(\IId)$.
For simplicity we assume that $ u \in \mathring{U}^{1/2}_\infty$. 
	
	For the nonadaptive approximation, according to Theorem \ref{thm-deep-NN1},  for any $\varepsilon > 0$ sufficient small 
	one can  explicitly construct a deep neural network architecture ${\mathbb A}_\varepsilon$ independent of $f$ and $a$, and  a deep ReLU neural network  $\Phi_\varepsilon(u)$ having the architecture ${\mathbb A}_\varepsilon$ such that	
	$$\|u-\Phi_\varepsilon(u)\|_\infty \leq \varepsilon,$$
	\begin{equation*}
	W({\mathbb A}_\varepsilon)\leq C d   \bigg( \frac{K_1^d}{(d-1)!}\bigg)^3 \varepsilon^{-2} \log(2\varepsilon^{-1})^{3(d-1)+1}, 
	\end{equation*}
	and
	$$
	L({\mathbb A}_\varepsilon) \leq C \log d \log(2\varepsilon^{-1}),
	$$
where $K_1:= 8(\sqrt 2 +1)^{3/2}$.

	For the adaptive approximation, according to Theorem \ref{thm-deep-NN},  for any $\varepsilon > 0$ sufficient small 
one can  explicitly construct an adaptive deep ReLU neural network $\Phi_\varepsilon(u)$ so that
\begin{equation*} 
\|u-\Phi_\varepsilon(u)\|_{\infty} \leq \varepsilon,
\end{equation*}
\begin{equation*}
W(\Phi_\varepsilon(u))\leq C d^2 \bigg(\frac{K_2^d}{(d-1)!}\bigg)^6
 \varepsilon^{-2}    \big(\log(2\varepsilon^{-1})\log  \log (2\varepsilon^{-1})\big)^{3(d-1)},
\end{equation*}
and
\begin{equation*} 
 L(\Phi_\varepsilon(u)) \leq C' \varepsilon^{-\frac{2}{d}} (\log (2\varepsilon^{-1}))^{\frac{2d-3}{d}}  (\log\log(2\varepsilon^{-1}))^{\frac{3(d-1)}{d}} ,
\end{equation*}
where $K_2 :=  16((2 + \sqrt 2))^{1/3}$.


\section{Conclusions}\label{sec:conclusion}
 We have presented both nonadaptive and adaptive methods for explicit construction of deep ReLU neural network $\Phi_\varepsilon(f)$ having an output that approximates functions  $f$ in the H\"older-Nikol'skii spaces with an arbitrary  prescribed accuracy $\varepsilon$ in the  $L_\infty$-norm. Nonadaptivity means that architecture of approximating deep ReLU neural networks is the same for all functions in $\Uas$. 
 For nonadaptive approximation, by using truncation of Faber series as a intermediate approximation, we have established  a dimension-dependent estimate for the computation complexity characterized by   the size $W(\Phi_\varepsilon(f))$ estimated by    
$$
	W(\Phi_\varepsilon(f))\leq C_\alpha d   \bigg( \frac{K_1^d}{(d-1)!}\bigg)^{\frac{1}{\alpha}+1} \varepsilon^{-\frac{1}{\alpha}} \log(2\varepsilon^{-1})^{(d-1)(\frac{1}{\alpha}+1)+1}, 
$$
where $K_1=B^{1/(\alpha+1)}4\alpha^{-1} $ with $B=(2^\alpha-1)^{-1}$. 

Concerning adaptive method,  for any $f\in \Uas$ we explicitly construct a deep ReLU neural network $\Phi_\varepsilon(f)$ of adaptive architecture having  the output  that approximates $f$ in  the $L_\infty(\IId)$-norm with a prescribed accuracy $\varepsilon$ and having   the size   estimated by 
$$
	W(\Phi_\varepsilon(f))\leq C_{\alpha} d^2\bigg(\frac{K_2^d}{(d-1)!}\bigg)^{\frac{2}{\alpha}+2} \varepsilon^{-
		\frac{1}{\alpha}}    \big(\log(2\varepsilon^{-1})\log  \log (2\varepsilon^{-1})\big)^{(1+\frac{1}{\alpha})(d-1)},
$$
where $K_2=4(2^{\alpha+3}B)^{\frac{1}{2\alpha+2}} (\alpha^{-1}\log(2\alpha^{-1}) )^{1/2}$. 

Construction of deep neural networks in the adaptive method is more involved but improves $\log (2\varepsilon^{-1})$  in the computation complexity of the approximating deep ReLU neural networks compared to the nonadaptive one. 

 Our theory is illustrated by an application to numerical approximation of the solution to elliptic PDEs.

\noindent
{\bf Acknowledgments.}  This work is funded by Vietnam National Foundation for Science and Technology Development (NAFOSTED) under  Grant No. 102.01-2020.03. A part of this work was done when  Dinh D\~ung and Van Kien Nguyen were working at the Vietnam Institute for Advanced Study in Mathematics (VIASM). They would like to thank  the VIASM  for providing a fruitful research environment and working condition.

\bibliographystyle{abbrv}

\bibliography{AllBib.bib}

\end{document}

%% file: Neuralnetwork-DK2-2020.tex
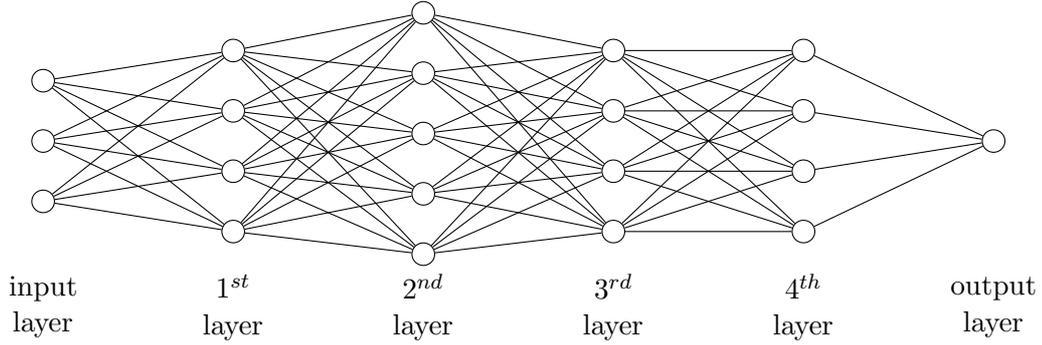
\begin{figure} 
	\begin{center}
		\begin{tikzpicture}  
		\tikzstyle{place}=[circle, draw=black,scale=1, inner sep=3pt,minimum size=1pt, align=center]
		
		\foreach \x in {1,...,3}
		\draw node at (0, -\x*0.8+1.6) [place] (input_\x) { $$};
		
		\foreach \x in {1,...,4}
		\node at (2.5, -\x*0.8+ 2.0) [place] (hidden1_\x){$$};

	\foreach \i in {1,...,3}
\foreach \j in {1,...,4}
\draw [-] (input_\i) to (hidden1_\j);

		\foreach \x in {1,...,5}
		\node at (5, -\x*0.8+2.5) [place] (hidden2_\x){$$};

\foreach \i in {1,...,4}
\foreach \j in {1,...,5}
\draw [-] (hidden1_\i) to (hidden2_\j);
		
		\foreach \x in {1,...,4}
		\node at (7.5, -\x*0.8+2.0) [place] (hiddenl1_\x){$$};

\foreach \i in {1,...,5}
\foreach \j in {1,...,4}
\draw [-] (hidden2_\i) to (hiddenl1_\j);

		\foreach \x in {1,...,4}
		\node at (10, -\x*0.8+2.0) [place] (hiddenl_\x){$$};

\foreach \i in {1,...,4}
\foreach \j in {1,...,4}
\draw [-] (hiddenl1_\i) to (hiddenl_\j);

	\foreach \x in {1}
\node at (12.5, 0) [place] (output_\x){$$};

		\foreach \i in {1,...,4}
		\foreach \j in {1}
		\draw [-] (hiddenl_\i) to (output_\j);

		\node at (0, -2.2) [align=center]{input \\ layer};
		\node at (2.5, -2.2) [align=center] {$1^{st}$\\ layer};
		\node at (5, -2.2) [align=center] {$2^{nd}$\\ layer};
		\node at (7.5, -2.2) [align=center] {$3^{rd}$\\ layer};
		\node at (10, -2.2) [align=center] {$4^{th}$\\ layer};
		\node at (12.5, -2.2) [align=center] {output\\ layer};

		
		\end{tikzpicture}
		\caption{The graph associated to a deep neural network with input dimension 3 and 5 layers}
		\label{fig:neuralnetwork}
	\end{center}
\end{figure}

%% file: Concatenation-DK2-2020.tex
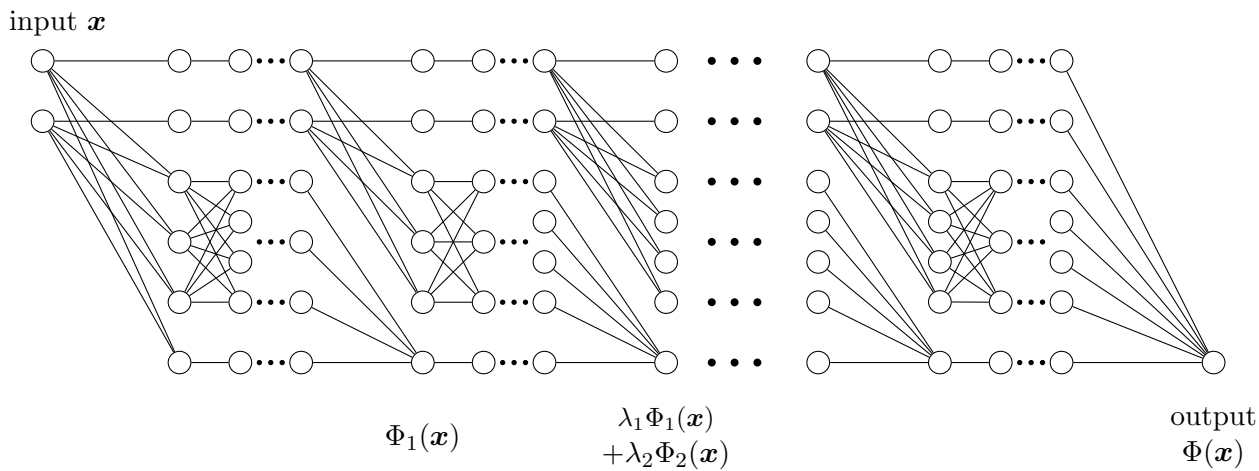
\begin{figure}
\begin{center}
\begin{tikzpicture}
\tikzstyle{place}=[circle, draw=black,scale=1, inner sep=3pt,minimum size=1pt, align=center]

\foreach \x in {1,...,2}
\draw node at (-0.2, -\x*0.8 +0.8) [place] (input_\x) { $$};
\node at (0, 0.5) {input $\bx$};

\foreach \x in {1,...,2}
\node at (1.6, -\x*0.8 +0.8) [place] (hiddenx1_\x){$$};

\foreach \x in {1,...,3}
\node at (1.6, -1.6 -\x*0.8+0.8) [place] (hiddenks1_\x){$$};

\node at (1.6, -4.0) [place] (hiddencol1){$$};

\foreach \x in {1,...,2}
\node at (2.4, -\x*0.8 +0.8) [place] (hiddenx15_\x){$$};

\foreach \x in {1,...,4}
\node at (2.4, -1.6 -\x*0.533+0.533) [place] (hiddenks15_\x){$$};

\node at (2.4, -4.0) [place] (hiddencol15){$$};

\foreach \x in {1,...,3}
\fill (2.5+\x*0.15, -0.0) circle (1pt);
\foreach \x in {1,...,3}
\fill (2.5+\x*0.15, -0.8) circle (1pt);
\foreach \x in {1,...,3}
\fill (2.5+\x*0.15, -1.6) circle (1pt);
\foreach \x in {1,...,3}
\fill (2.5+\x*0.15, -2.4) circle (1pt);
\foreach \x in {1,...,3}
\fill (2.5+\x*0.15, -3.2) circle (1pt);
\foreach \x in {1,...,3}
\fill (2.5+\x*0.15, -4.0) circle (1pt);

\foreach \x in {1,...,2}
\node at (3.2,  -\x*0.8 +0.8) [place] (hiddenx2_\x){$$};

\foreach \x in {1,...,3}
\node at (3.2, -1.6 -\x*0.8+0.8) [place] (hiddenks2_\x){$$};

\node at (3.2, -4.0) [place] (hiddencol2){$$};

\foreach \x in {1,...,2}
\node at (4.8, -\x*0.8 +0.8) [place] (hiddenx3_\x){$$};
\foreach \x in {1,...,3}
\node at (4.8, -1.6 -\x*0.8+0.8) [place] (hiddenks3_\x){$$};
\node at (4.8, -4.0) [place] (hiddencol3){$$};
\node at (4.8, -5.0)  {$\Phi_{1}(\bx)$};

\foreach \x in {1,...,2}
\node at (5.6,  -\x*0.8 +0.8) [place] (hiddenx35_\x){$$};
\foreach \x in {1,...,3}
\node at (5.6, -1.6 -\x*0.8+0.8) [place] (hiddenks35_\x){$$};
\node at (5.6, -4.0) [place] (hiddencol35){$$};

\foreach \x in {1,...,3}
\fill (5.7+\x*0.15, -0.0) circle (1pt);
\foreach \x in {1,...,3}
\fill (5.7+\x*0.15, -0.8) circle (1pt);
\foreach \x in {1,...,3}
\fill (5.7+\x*0.15, -1.6) circle (1pt);
\foreach \x in {1,...,3}
\fill (5.7+\x*0.15, -2.4) circle (1pt);
\foreach \x in {1,...,3}
\fill (5.7+\x*0.15, -3.2) circle (1pt);
\foreach \x in {1,...,3}
\fill (5.7+\x*0.15, -4.0) circle (1pt);

\foreach \x in {1,...,2}
\node at (6.4,  -\x*0.8 +0.8) [place] (hiddenx4_\x){$$};
\foreach \x in {1,...,4}
\node at (6.4, -1.6 -\x*0.533+0.533) [place] (hiddenks4_\x){$$};
\node at (6.4, -4.0) [place] (hiddencol4){$$};

\foreach \x in {1,...,2}
\node at (8.0,  -\x*0.8 +0.8) [place] (hiddenx5_\x){$$};
\foreach \x in {1,...,4}
\node at (8.0, -1.6 -\x*0.533+0.533) [place] (hiddenks5_\x){$$};
\node at (8.0, -4.0) [place] (hiddencol5){$$};
\node at (8.0, -5.0) [align=center] {$\small \lambda_1\Phi_{1}(\bx)$
\\
$+ \lambda_2\Phi_{2}(\bx)$};

\foreach \x in {1,...,3}
\fill (8.3+\x*0.3, -0.0) circle (1.5pt);
\foreach \x in {1,...,3}
\fill (8.3+\x*0.3, -0.8) circle (1.5pt);
\foreach \x in {1,...,3}
\fill (8.3+\x*0.3, -1.6) circle (1.5pt);
\foreach \x in {1,...,3}
\fill (8.3+\x*0.3, -2.4) circle (1.5pt);
\foreach \x in {1,...,3}
\fill (8.3+\x*0.3, -3.2) circle (1.5pt);
\foreach \x in {1,...,3}
\fill (8.3+\x*0.3, -4.0) circle (1.5pt);

\foreach \x in {1,...,2}
\node at (10.0, -\x*0.8 +0.8) [place] (hiddenx6_\x){$$};
\foreach \x in {1,...,4}
\node at (10.0, -1.6 -\x*0.533+0.533) [place] (hiddenks6_\x){$$};
\node at (10.0, -4.0) [place] (hiddencol6){$$};

\foreach \x in {1,...,2}
\node at (11.6, -\x*0.8 +0.8) [place] (hiddenx7_\x){$$};
\foreach \x in {1,...,4}
\node at (11.6, -1.6 -\x*0.532+0.532) [place] (hiddenks7_\x){$$};
\node at (11.6, -4.0) [place] (hiddencol7){$$};

\foreach \x in {1,...,2}
\node at (12.4, -\x*0.8 +0.8) [place] (hiddenx75_\x){$$};
\foreach \x in {1,...,3}
\node at (12.4, -1.6 -\x*0.8+0.8) [place] (hiddenks75_\x){$$};
\node at (12.4, -4.0) [place] (hiddencol75){$$};
\foreach \x in {1,...,3}
\fill (12.5+\x*0.15, -0.0) circle (1pt);
\foreach \x in {1,...,3}
\fill (12.5+\x*0.15, -0.8) circle (1pt);
\foreach \x in {1,...,3}
\fill (12.5+\x*0.15, -1.6) circle (1pt);
\foreach \x in {1,...,3}
\fill (12.5+\x*0.15, -2.4) circle (1pt);
\foreach \x in {1,...,3}
\fill (12.5+\x*0.15, -3.2) circle (1pt);
\foreach \x in {1,...,3}
\fill (12.5+\x*0.15, -4.0) circle (1pt);

\foreach \x in {1,...,2}
\node at (13.2, -\x*0.8 +0.8) [place] (hiddenx8_\x){$$};
\foreach \x in {1,...,4}
\node at (13.2, -1.6 -\x*0.533+0.533) [place] (hiddenks8_\x){$$};
\node at (13.2, -4.0) [place] (hiddencol8){$$};

\node at (15.2, -4.0)  [place] (output){$$};
\node at (15.2, -5.0) [align=center] {output
\\
$\Phi(\bx)$
};

\foreach \i in {1,...,2}
\foreach \j in {1,...,3}
\draw [-] (input_\i) to (hiddenks1_\j);
\foreach \i in {1,...,2}
\draw [-] (input_\i) to (hiddencol1);
\foreach \i in {1,...,2}
\draw [-] (input_\i) to (hiddenx1_\i);


\foreach \i in {1,...,2}
\draw [-] (hiddenx1_\i) to (hiddenx15_\i);

\foreach \i in {1,...,3}
\foreach \j in {1,...,4}
\draw [-] (hiddenks1_\i) to (hiddenks15_\j);
\draw [-] (hiddencol1) to (hiddencol15);



\foreach \i in {1,...,2}
\draw [-] (hiddenx2_\i) to (hiddenx3_\i);
\foreach \i in {1,...,2}
\foreach \j in {1,...,3}
\draw [-] (hiddenx2_\i) to (hiddenks3_\j);

\foreach \j in {1,...,3}
\draw [-] (hiddenks2_\j) to (hiddencol3);
\draw [-] (hiddencol2) to (hiddencol3);


\foreach \i in {1,...,2}
\draw [-] (hiddenx3_\i) to (hiddenx35_\i);
\foreach \i in {1,...,3}
\foreach \j in {1,...,3}
\draw [-] (hiddenks3_\i) to (hiddenks35_\j);

\draw [-] (hiddencol3) to (hiddencol35);



\foreach \i in {1,...,2}
\draw [-] (hiddenx4_\i) to (hiddenx5_\i);
\foreach \i in {1,...,2}
\foreach \j in {1,...,4}
\draw [-] (hiddenx4_\i) to (hiddenks5_\j);

\foreach \j in {1,...,4}
\draw [-] (hiddenks4_\j) to (hiddencol5);
\draw [-] (hiddencol4) to (hiddencol5);



\foreach \i in {1,...,2}
\draw [-] (hiddenx6_\i) to (hiddenx7_\i);
\foreach \i in {1,...,2}
\foreach \j in {1,...,4}
\draw [-] (hiddenx6_\i) to (hiddenks7_\j);

\foreach \j in {1,...,4}
\draw [-] (hiddenks6_\j) to (hiddencol7);
\draw [-] (hiddencol6) to (hiddencol7);


\foreach \i in {1,...,2}
\draw [-] (hiddenx7_\i) to (hiddenx75_\i);
\foreach \i in {1,...,4}
\foreach \j in {1,...,3}
\draw [-] (hiddenks7_\i) to (hiddenks75_\j);

\draw [-] (hiddencol7) to (hiddencol75);


\foreach \i in {1,...,2}
\draw [-] (hiddenx8_\i) to (output);
\foreach \j in {1,...,4}
\draw [-] (hiddenks8_\j) to (output);

\draw [-] (hiddencol8) to (output);
\end{tikzpicture}
\caption{Illustration of a special deep ReLU neural network ($d=2$)}
\label{fig:concatenation}
\end{center}
\end{figure}